\documentclass[final,onefignum,onetabnum]{siamonline190516}

\usepackage{lipsum}
\usepackage{amsfonts}
\usepackage{graphicx}
\usepackage{epstopdf}
\usepackage{algorithmic}
\ifpdf
  \DeclareGraphicsExtensions{.eps,.pdf,.png,.jpg}
\else
  \DeclareGraphicsExtensions{.eps}
\fi

\usepackage{enumitem}
\setlist[enumerate]{leftmargin=.5in}
\setlist[itemize]{leftmargin=.5in}

\newsiamremark{remark}{Remark}
\newsiamremark{hypothesis}{Hypothesis}
\newsiamthm{prop}{Proposition}
\newsiamthm{obs}{Observation}
\newsiamthm{defin}{Definition}
\newsiamthm{problem}{Problem}

\headers{Wasserstein barycenters are NP-hard to compute}{Jason M. Altschuler and Enric Boix-Adser\`a}

\title{Wasserstein barycenters are NP-hard to compute\thanks{This work was partially supported by NSF Graduate Research Fellowship 1122374, a Siebel PhD Fellowship, and a TwoSigma PhD fellowship.}}

\author{
	Jason M. Altschuler\thanks{Laboratory for Information and Decision Systems (LIDS), Massachusetts Institute of Technology, Cambridge MA 02139 (\email{jasonalt@mit.edu}, \email{eboix@mit.edu}).}
		\and
	Enric Boix-Adser\`a\footnotemark[2]}

\usepackage{amsopn}

\usepackage[square,sort,numbers]{natbib}
\usepackage[sort,nocompress,noadjust]{cite}

\usepackage{amsmath}
\usepackage{amssymb}
\usepackage{MnSymbol} 
\usepackage{graphicx,color}
\usepackage{dsfont}
\usepackage{booktabs}
\usepackage[normalem]{ulem}
\usepackage{mathtools}
\usepackage[nameinlink,capitalize]{cleveref}
\usepackage{xspace}
\usepackage{tikz}

\numberwithin{equation}{section}

\newcommand{\cM}{\mathcal{M}}

\newcommand{\cW}{\mathcal{W}}

\newcommand{\R}{\mathbb{R}}

\newcommand{\N}{\mathbb{N}}

\newcommand{\Rn}{\R^n}

\newcommand{\Rntk}{(\R^n)^{\otimes k}}
\newcommand{\Rpntk}{(\R_{\geq 0}^n)^{\otimes k}}

\newcommand*{\E}{\mathbb{E}}

\newcommand*{\poly}{\mathrm{poly}}

\newcommand*{\eps}{\varepsilon}

\DeclareMathOperator*{\argmin}{argmin}

\renewcommand{\leq}{\leqslant}
\renewcommand{\geq}{\geqslant}

\providecommand{\abs}[1]{\left\lvert#1\right\rvert}
\providecommand{\set}[1]{\{#1\}}

\newcommand{\MOT}{\textsf{MOT}}

\newcommand{\Cmax}{C_{\max}}

\newcommand{\Coup}{\cM(\mu_1,\dots,\mu_k)}

\DeclareMathOperator{\optbary}{OPT_{BARY}}

\newcommand{\jvec}{\vec{j}}
\newcommand{\minprob}{\min_{\jvec \in [n]^k} C_{\jvec}}

\DeclarePairedDelimiter{\ceil}{\lceil}{\rceil}

\let\baraccent=\= %
\renewcommand{\=}[1]{\stackrel{#1}{=}} %

\providecommand{\tsc}{\textsc}
\providecommand{\RR}{\mathbb{R}}

\providecommand{\cM}{\mathcal{M}}

\providecommand{\PP}{\mathbb{P}}

\providecommand{\eps}{\varepsilon}

\providecommand{\N}{\mathbb{N}}

\renewcommand{\P}{\mathsf{P}}
\providecommand{\NP}{\mathsf{NP}}

\providecommand{\BPP}{\mathsf{BPP}}

\mathchardef\mhyphen="2D %

\providecommand{\sm}{\setminus}

\newcommand{\interior}[1]{%
	{\kern0pt#1}^{\mathrm{o}}%
}

\usepackage{fancybox}
\newenvironment{fminipage}{\begin{Sbox}\begin{minipage}}{\end{minipage}\end{Sbox}\fbox{\TheSbox}}

\providecommand{\clique}{\tsc{Clique}}

\providecommand{\chub}{\tsc{Cheapest-Hub}}
\providecommand{\chubpq}{\tsc{Cheapest-Hub}_{p,q}}

\newcommand{\val}{F}

\ifpdf
\hypersetup{
  pdftitle={Wasserstein barycenters are NP-hard to compute},
  pdfauthor={Jason M. Altschuler \and Enric Boix-Adser\`a}
}
\fi

\begin{document}

\maketitle

\begin{abstract}
	Computing Wasserstein barycenters (a.k.a. Optimal Transport barycenters) is a fundamental problem in geometry which has recently attracted considerable attention due to many applications in data science. While there exist polynomial-time algorithms in any fixed dimension, all known running times suffer exponentially in the dimension. It is an open question whether this exponential dependence is improvable to a polynomial dependence. This paper proves that unless $\P = \NP$, the answer is no. This uncovers a ``curse of dimensionality'' for Wasserstein barycenter computation which does not occur for Optimal Transport computation. Moreover, our hardness results for computing Wasserstein barycenters extend to approximate computation, to seemingly simple cases of the problem, and to averaging probability distributions in other Optimal Transport metrics.
\end{abstract}

\section{Introduction}\label{sec:intro}

Wasserstein barycenters provide a natural approach for averaging probability distributions in a way that respects their geometry. In words, Wasserstein barycenters are the Riemannian centers of mass (a.k.a. Fr\'echet means) with respect to the Optimal Transport distance~\citep{AguCar11}. More precisely, given probability distributions $\mu_1, \dots, \mu_k$ over $\R^d$ and non-negative weights $\lambda_1, \dots, \lambda_k$ summing to $1$, the corresponding Wasserstein barycenters are the probability distributions $\nu$ over $\R^d$ that minimize
\begin{align}
	\min_{\nu} \sum_{i=1}^k \lambda_i \cW^2(\mu_i,\nu).\label{eq:bary}
\end{align}
Here, $\cW$ denotes the $2$-Wasserstein distance (a.k.a. the standard Optimal Transport distance) between probability distributions~\citep{Vil03}, which we recall is defined as
\begin{align}
	\cW(\mu,\nu) = \left( \inf_{\pi \in \cM(\mu,\nu)} 
	\E_{(X,Y) \sim \pi} \|X-Y\|_2^2 
	\right)^{1/2},
	\nonumber
\end{align}
where $\cM(\mu,\nu)$ is the set of joint distributions with first marginal $\mu$ and second marginal $\nu$.

Wasserstein barycenters have received considerable research attention over the past decade due to their elegant mathematical properties (see, e.g.,~\citep{AguCar11}) and many data-science applications (see, e.g., the surveys~\citep{PeyCut17,panaretos2019statistical}).
For example, illustrative applications include improving Bayesian learning by averaging posterior distributions~\citep{SriLiDun18}, improving sensors by averaging their measurements~\citep{elvander2020multi}, interpolating between shapes by averaging them (viewed as point clouds in Euclidean space)~\citep{solomon2015convolutional}, clustering documents (viewed as distributions over word embeddings)~\citep{ye2017fast,ye2017determining}, multilevel clustering of datasets~\citep{ho2017multilevel,ho2019probabilistic}, and unsupervised representation learning in natural language processing~\citep{singh2020context}. Note that some of these applications are in low-dimensional settings (e.g., graphics, imaging, and physical applications), while others are in high-dimensional settings (e.g., natural language processing, machine learning, and statistics applications). 

A key issue that determines how useful Wasserstein barycenters are in applications is whether they can be computed efficiently. Note that in most computational applications, each measure $\mu_i$ is a discrete distribution: it is a ``point cloud'' over data points. This motivates the following fundamental question, which has remained open despite considerable research attention (see the previous work section). 
\begin{align*}
	\text{\emph{Can Wasserstein barycenters of discrete distributions be computed in polynomial time?}}
\end{align*}
That is, can the optimization problem~\eqref{eq:bary} be solved in time that is polynomial in the number of distributions $k$, the dimension $d$, the maximum support size $n$ of the input distributions $\mu_i$, and the bit complexity $\log U$ of each entry in the input measures and weights? This constitutes a running time that is polynomial in the input size since each discrete measure is naturally described as a list of at most $n$ point locations and the corresponding probability masses. 

\par Recent work has shown that Wasserstein barycenters can in fact be computed in polynomial time in any \emph{fixed} dimension $d$~\citep{AltBoi20bary}. However, the dependence of this algorithm is exponential in $d$, and it is unclear whether this can be improved.\footnote{In the special case of constant $k$, it is possible to avoid the curse of dimensionality by simply using standard LP solvers to solve a well-known Multimarginal Optimal Transport reformulation of the problem~\citep{carlier2010matching,anderes2016discrete,BenCarCut15,AguCar11}. This takes $\poly(n^k,d) = \poly(n,d)$ time for constant $k$.} Exponential dependence in $d$ prohibits scalability of Wasserstein barycenter computation to high-dimensional settings often encountered in data-science applications such as clustering~\citep{ho2017multilevel,ho2019probabilistic,ye2017fast,ye2017determining}, representation learning~\citep{singh2020context}, and more~\citep{xu2018distilled,cohen2020estimating}. Thus it is of both practical and theoretical significance to understand whether there exists a provably accurate and efficient algorithm for computing high-dimensional Wasserstein barycenters.

\subsection{Contributions}\label{ssec:intro:cont}

Here we outline our main results; our techniques are described in the following subsection.

\subsubsection{Computational complexity of Wasserstein barycenters}

This paper resolves the above question: unless $\P = \NP$, the answer is no. This result uncovers a ``curse of dimensionality'' for computing Wasserstein barycenters in light of the aforementioned result that they are polynomial-time computable in any fixed dimension~\citep{AltBoi20bary}. This is perhaps surprising because there is no curse of dimensionality for Optimal Transport computation (that problem is well-known to be polynomial-time computable, see, e.g.,~\citep{schrijver2003combinatorial}). 
This explains why---despite a rapidly growing literature (see the previous work section)---there has been a lack of progress towards developing algorithms that provably compute optimal Wasserstein barycenters in polynomial time.

\begin{theorem}[NP-hardness]\label{thm:np}
	Assuming $\P \neq \NP$, there is no algorithm that, given distributions $\mu_1, \dots, \mu_k$ and uniform weights $\lambda_1, \dots, \lambda_k = 1/k$,
	computes the value of the Wasserstein barycenter problem~\eqref{eq:bary} in $\poly(n,k,d,\log U)$ time.
\end{theorem}

\par Moreover, this hardness extends even to computation of \emph{approximate} Wasserstein barycenters. This extension requires the slightly stronger yet standard complexity-theoretic assumption $\NP \not\subset \BPP$, which in words is the statement that $\NP$-hard problems do not admit polynomial-time randomized algorithms; see the preliminaries for a formal definition. This result is formally stated as follows. Below, let $R$ be an upper bound on the squared diameter of the supports of the measures $\mu_i$; any running time must depend on $R$ and the accuracy $\eps$ through the scale-invariant ratio $R/\eps$. 

\begin{theorem}[Inapproximability]\label{thm:inapprox}
	Assuming $\NP \not\subset \BPP$, there is no randomized algorithm that, given distributions $\mu_1, \dots, \mu_k$ and uniform weights $\lambda_1, \dots, \lambda_k = 1/k$, computes the value of the Wasserstein barycenter problem~\eqref{eq:bary} to $\eps$ additive accuracy with probability at least $2/3$ in $\poly(n,k,d,\log U, R/\eps)$ time.
\end{theorem}

We make two remarks about these results. First, since Theorems~\ref{thm:np} and~\ref{thm:inapprox} establish hardness for (approximately) computing the optimal \emph{value} of the barycenter problem, as an immediate corollary they preclude finding an (approximately) optimal \emph{solution}. Specifically, since $\cW(\nu,\mu_i)$ is computable in polynomial time (e.g., via linear programming~\citep{schrijver2003combinatorial}) whenever $\nu$ has polynomial-size support, these results imply that: unless $\P = \NP$, there is no polynomial-time algorithm for computing a barycenter $\nu$ with polynomial-size support.\footnote{Note that there always \emph{exists} a barycenter with support of size $O(nk)$~\citep{anderes2016discrete}. Our result shows that if $\P \neq \NP$, then such a barycenter cannot be efficiently computed.}

\par Second, these hardness results hold even in seemingly simple settings. For example, our results are written for the case where all weights $\lambda_1 = \dots = \lambda_k  = 1/k$ are uniform. Our construction also sets all measures $\mu_1, \dots, \mu_k$ to be supported on $n$ points each with all support points having $\{0,1\}$-valued coordinates\footnote{For the generalized Wasserstein barycenter problem in the cases $q \in \{1,\infty\}$, our construction has points with $\{-1,0,1\}$-valued coordinates.}, and can be readily extended to the case where $\mu_1, \dots, \mu_k$ are uniform distributions (Theorem~\ref{thm:inapproxuniform}).

\subsubsection{Computational complexity of generalized Wasserstein barycenters}

We further demonstrate that our $\NP$-hardness results capture a robust phenomenon about averaging high-dimensional distributions by showing that these results extend to other important notions of averaging in Wasserstein space that are studied in the literature, see, e.g.,~\citep{CutDou14,CarObeOud15,gramfort2015fast,lin2020fixed}. Specifically, we show that for any fixed $p \in [1,\infty)$ and $q \in [1,\infty]$, the following ``generalized Wasserstein barycenter problem'' is similarly computationally hard: 
\begin{align}
	\min_{\nu} \sum_{i=1}^k \lambda_i \cW_{p,q}^p(\mu_i,\nu),
	\label{eq:baryext}
\end{align}
where $\cW_{p,q}$ is the $p$-Wasserstein distance over the metric space $(\R^d, \ell_q)$, i.e.,
\begin{align}
	\cW_{p,q}(\mu,\nu) = \left( \inf_{\pi \in \cM(\mu,\nu)} 
	\E_{(X,Y) \sim \pi} \|X-Y\|_q^p 
	\right)^{1/p}.
	\nonumber
\end{align}
Clearly this generalized Wasserstein barycenter problem~\eqref{eq:baryext} captures the standard Wasserstein barycenter problem~\eqref{eq:bary} in the case that $p=q=2$ (we use the shorthand $\cW$ for $\cW_{2,2}$). The reason that this problem is often studied in this generality is two-fold. First, the degree of freedom $p \in [1,\infty)$ enables adjusting the notion of average to be more or less affected by outliers. For instance, when $p=1$ this problem recovers ``Wasserstein geometric medians'' which are known to be robust to outliers (e.g., they have a breakdown point of 50\%~\citep{FKJ}), whereas when $p$ tends to $\infty$ this problem becomes finding $\nu$ that is ``fair'' to all input measures in that its distance is small to all simultaneously. The flexibility to choose $p$ based on the downstream application is a common desiderata in robust statistics, as has been argued ever since the influential paper of Fr\'echet in 1948~\citep{frechet1948elements}. Second, the degree of freedom $q \in [1,\infty]$ enables handling applications where the geometry is non-Euclidean since this corresponds to measuring the transportation cost via any $\ell_q$ norm.

\par Formally, we show the following inapproximability result for this generalized Wasserstein problem~\eqref{eq:baryext} for the entire range of possible parameters $p \in [1,\infty)$ and $q\in [1,\infty]$. Note that we only show hardness of approximation here because exact computation is hard for trivial reasons: for general $p,q \neq 2$, the bit complexity of the optimal value~\eqref{eq:baryext} might not be polynomially bounded in the input size (e.g., it could be the case that the inputs are rational but the optimal value is irrational). In what follows, the measures $\mu_1, \dots, \mu_k$ are still each over $n$ atoms in $\R^d$. However, now the appropriate quantity generalizing $R$ is $R_{p,q}$, the $p$-th power of the $\ell_q$ norm diameter of these supports; this ensures that $R_{p,q}/\eps$ is scale-invariant.

\begin{theorem}[Generalization of Theorem~\ref{thm:inapprox}]\label{thm:inapproxext}
	Fix $p \in [1, \infty)$ and $q \in [1, \infty]$. Assuming $\NP \not\subset \BPP$, there does not exist a randomized algorithm that, given distributions $\mu_1, \dots, \mu_k$ and uniform weights $\lambda_1, \dots, \lambda_k = 1/k$, computes the value of the $(p,q)$-Wasserstein barycenter problem~\eqref{eq:baryext} to $\eps$ additive accuracy with probability at least $2/3$ in $\poly(n,k,d,\log U, R_{p,q}/\eps)$ time.
\end{theorem}

\subsection{Techniques}\label{ssec:intro:tech}

\paragraph*{Starting point: reduction from $\chub$ to Wasserstein barycenters} Our starting point is the combination of two known results. The first result is the equivalence of the generalized Wasserstein barycenter problem~\eqref{eq:baryext} and a Multimarginal Optimal Transport ($\MOT$) problem with a particular cost tensor~\citep{carlier2010matching,anderes2016discrete,BenCarCut15,AguCar11}.
Note that while certain $\MOT$ problems are known to be $\NP$-hard~\citep{AltBoi20mothard}, the relevant $\MOT$ cost tensor here does not fall under any known such classifications, and thus the computational complexity of this problem does not follow from previous work.
The second result is a toolkit recently developed in~\citep{AltBoi20mothard} for understanding the computational complexity of an $\MOT$ problem based on its cost tensor. See the preliminaries section \S\ref{sec:prelim} for details on these two results. Together, they imply that to prove hardness of (approximating) the generalized Wasserstein barycenter problem~\eqref{eq:baryext}, it suffices to prove hardness of (approximating) the following problem---which we call $\chub$.

\begin{definition}[$\chubpq$]\label{def:chub}
	Given points $\{x_{i,j}\}_{i \in [k], j \in [n]} \subset \R^d$ as input, the $\chubpq$ problem is to compute 
	\[
	\min_{(j_1, \dots, j_k) \in [n]^k} 
	F_{p,q}(x_{1,j_1}, \dots, x_{k,j_k}),
	\]
	where
	\begin{align}
		F_{p,q}(z_1, \dots, z_k) := 
		\min_{y \in \R^d} 
		\sum_{i=1}^k 
		\|z_i - y\|_q^p.
		\label{eq:F-def}
	\end{align}
\end{definition}

\paragraph*{Showing $\NP$-hardness of approximating $\chub$}
This is our main technical contribution. As is typical with $\NP$-hardness proofs, the first challenge is to determine what $\NP$-hard problem to reduce from. The motivation behind our reduction is the geometric interpretation of the $\chub$ problem (which also explains our naming of $\chub$). This geometric interpretation is: given $k$ sets $S_i \subset \R^d$, each consisting of $n$ points $S_i = \{x_{i,j}\}_{j \in [n]}$, find one point from each set so as to minimize the average distance (measured in the relevant geometry via the function $F_{p,q}$) to the closest ``hub'' $y \in \R^d$. See Figure~\ref{fig:chub} for an illustration. On an intuitive level, this question of finding $k$ points which are close somewhat resembles the $k$-$\clique$ problem, which is the task of finding a set of $k$ vertices in a graph such that all pairs of vertices are close in the sense of being adjacent. Motivated by this intuitive observation, we show a reduction from $k$-$\clique$ to $\chub$. 

\begin{figure}
	\begin{center}
		\includegraphics[scale=0.25,trim={400 300 600 300},clip]{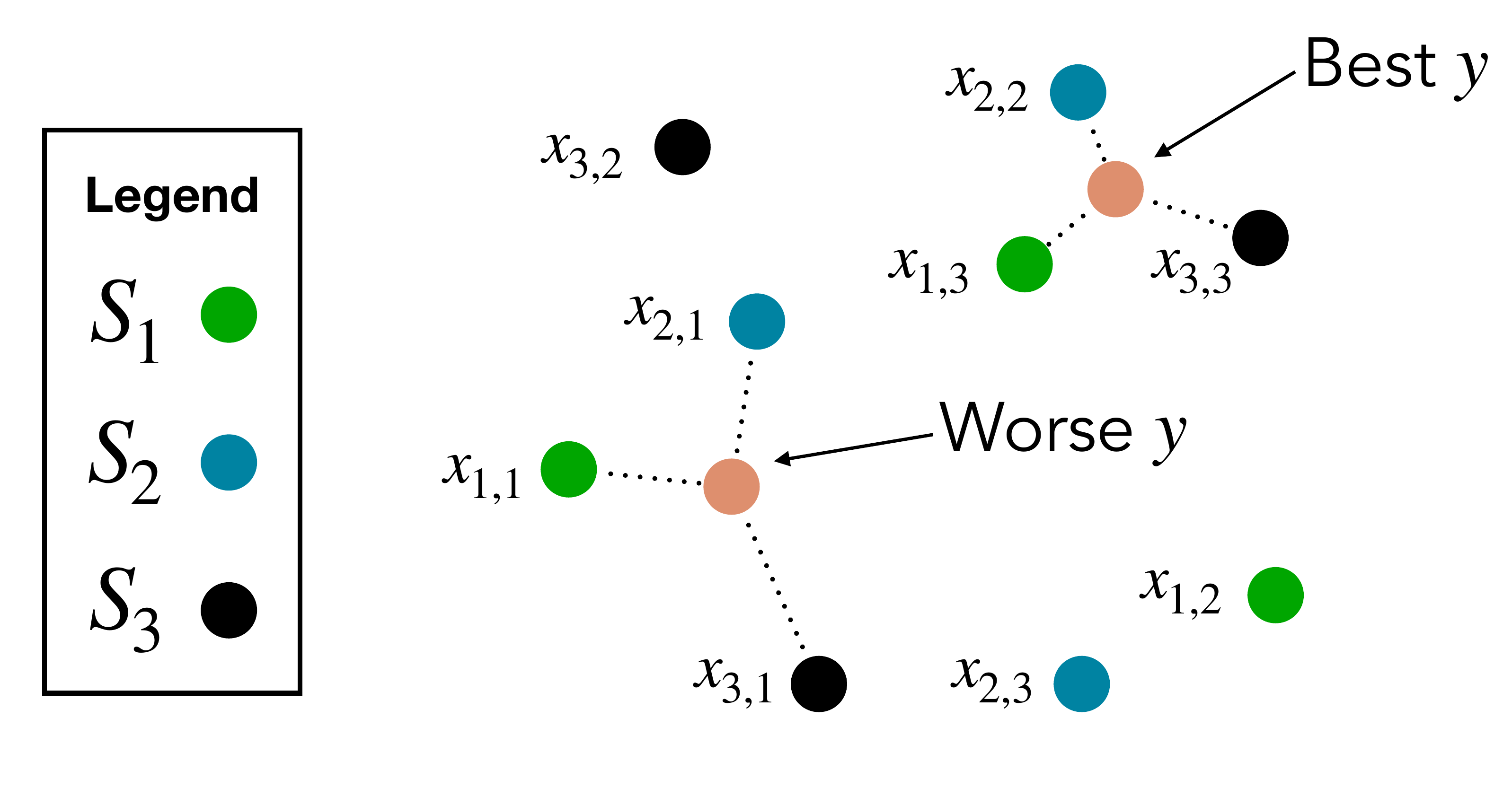}
	\end{center}
	\caption{Illustration of the $\chub$ problem in the case of $k=3$ sets, each consisting of $n=3$ points. The points in set $S_i$ are denoted by $\{x_{i,1}, x_{i,2}, x_{i,3}\}$ and displayed in the same color. The $\chub$ problem is to choose one point per set so as to minimize the average distance (measured in the relevant geometry by the function $F_{p,q}$) to the closest ``hub'' $y \in \R^d$. Top: the points $\{x_{1,3}, x_{2,2}, x_{3,3}\}$ corresponding to tuple $(j_1,j_2,j_3) = (3,2,3)$ yield the best hub $y$. Bottom: the points $\{x_{1,1}, x_{2,1}, x_{3,1}\}$ corresponding to tuple $(j_1,j_2,j_3) = (1,1,1)$ yield a suboptimal hub. 
	}
	\label{fig:chub}
\end{figure}

\par A key part of this reduction is figuring out how to appropriately embed the (combinatorial) adjacency properties of a graph $G$ into a (geometric) point configuration. Briefly, we show that, given an $n$-vertex graph $G = (V,E)$, one can efficiently compute points $\{x_{i,j}\}_{i \in [k], j \in [n]} \subset \R^d$ such that the value of the corresponding $\chub$ problem indicates whether $G$ has a clique of size $k$. Our embedding ensures that $G$ has a clique of size $k$ if and only if there are $k$ points $x_{1,j_1}, \dots, x_{k,j_k}$ that are sufficiently close to each other. Roughly speaking, we achieve this by setting the $nk$ points $\{x_{i,j}\}_{i \in [k], j \in [n]}$ to be an embedding of $k$ copies of the vertex set $V$ of the graph $G$, where adjacent vertices are embedded as close points in $\R^d$.

\paragraph*{Intuition for the special case of $p=q=2$ (the standard Wasserstein barycenter)} For concreteness, let us explain our proof in the case $p=q=2$ (a.k.a. the standard Wasserstein barycenter). The general case of $p \in [1,\infty)$ and $q \in [1,\infty]$ follows a similar high-level approach but is significantly more involved, as described below.

\par In this case, the minimization over $y$ in $\chub_{2,2}$ has a simple closed-form. By direct calculation\footnote{This equivalence requires all points $x_{i,j}$ to have the same norm, which our construction ensures.}, the $\chub_{2,2}$ problem is equivalent to the problem of finding $k$ maximally correlated vectors, one from each of the $k$ sets $S_i = \{x_{i,j}\}_{j \in [n]}$, i.e.,
\begin{align}
	\max_{(j_1, \dots, j_k) \in [n]^k} \sum_{i < i' \in [k]} \langle x_{i,j_i}, x_{i',j_{i'}} \rangle.
	\label{eq:tech-22:2}
\end{align}
Our embedding is based off the following observation. Consider $x_{i,j} \in \{0,1\}^{|E|}$ to be the edge-indicator vector of vertex $j \in [n] \cong V$, that is, has $e$-th entry equal to $1$ if vertex $j$ is an endpoint of edge $e$. Then $\langle x_{i,j_i},x_{i',j_{i'}} \rangle = \mathds{1}[\text{vertices }j_i\text{ and }j_{i'}\text{ are adjacent}]$ for any pair of distinct indices $j_i$ and $j_{i'}$, hence 
\begin{align*}
	\max_{\substack{(j_1, \dots, j_k) \in [n]^k \\ \text{s.t. } j_i \neq j_{i'}, \; \forall i \neq i'}} \sum_{i < i' \in [k]} \langle x_{i,j_i}, x_{i',j_{i'}} \rangle
	&=
	\max_{\substack{(j_1, \dots, j_k) \in [n]^k \\ \text{s.t. } j_i \neq j_{i'}, \; \forall i \neq i'}}
	\#\left( \text{edges between the vertices } j_1, \dots, j_k \right)
\end{align*}
is equal to $\binom{k}{2}$ if $G$ contains a $k$-clique, or otherwise is at most $\binom{k}{2} - 1$. Therefore if the optimization problem~\eqref{eq:tech-22:2} were restricted to tuples $(j_1,\dots,j_k) \in [n]^k$ with \emph{distinct} entries, then this would suffice for the reduction because a maximizing tuple (up to any error less than $1/2$) would yield a maximum clique size. Dealing with non-distinctness requires a more careful embedding into higher ambient dimension in which $x_{i,j}$ and $x_{i',j}$ are far from each other rather than identical when $i \neq i'$; details in \S\ref{ssec:bary:proof}.

\paragraph*{Obstacles for general case of $p \in [1,\infty)$ and $q \in [1,\infty]$} Although our proofs for the cases beyond $p=q=2$ exploit the same intuitive connection between $\chub$ and $\clique$, these proofs are significantly more involved and in some cases require altogether different embeddings. A core difficulty is that unlike the $p=q=2$ case, in general there is no closed-form solution for the minimization over $y$ in the $\chub$ optimization problem. Thus we cannot analytically compute the value $F_{p,q}(x_{1,j_1}, \dots, x_{k,j_k})$ and from that argue that this value is small or not depending on whether the set of vertices $\{j_1,\dots,j_k\} \subset V$ is a $k$-clique in $G$. In fact, for general $p$ and $q$, the value $F_{p,q}(x_{1,j_1}, \dots, x_{k,j_k})$ is not even determined by the number of edges between the vertices $\{j_1, \dots, j_k\}$. 

At a high level, we overcome this obstacle via an inductive argument on the number of edges in the induced subgraph formed by the vertices $\{j_1, \dots, j_k\}$. Specifically, our key lemma states that for a fixed tuple $(j_1,\dots,j_k) \in [n]^k$, the value $F_{p,q}(x_{1,j_1},\dots,x_{k,j_k})$ significantly decreases if $G$ is changed in a way that ``adds an edge'' to this subgraph. By iteratively applying this lemma (and also separately showing that all $k$-cliques admit the same value), we conclude that $k$-cliques have a significantly lower value than non-$k$-cliques. Note, however, that proving this key lemma again runs into the issue of a lack of closed-form solution to the problem defining $F_{p,q}$, but by arguing about the specific local update to the decision variable we are able to reason about how the value of this convex optimization changes at each inductive step.

A further challenge that should also be mentioned is that different embeddings are needed in the cases of $q \in \{1,\infty\}$. Indeed, in both cases, the value $F_{p,q}(x_{1,j_1}, \dots, x_{k,j_k})$ is a \emph{constant} independent of the number of edges between the vertices $\{j_1, \dots, j_k\}$ if one uses the same embedding to construct the points $\{x_{i,j}\}$ as we do in the $q \in (1,\infty)$ case. See \S\ref{app:q1} and \S\ref{app:qinf} for details on the embeddings needed for these cases.

\subsection{Related work}
The many applications of Wasserstein barycenters have motivated an extensive literature that approaches this problem from both the algorithmic and hardness sides. Here we contextualize our results with the literature.

\subsubsection{Algorithms for the Wasserstein barycenter problem}
Many algorithms have been proposed. However, all of them have running time which scales exponentially in at least one of the input parameters, and/or do not provably compute arbitrarily close approximations, described below. The purpose of this paper is to show that this is unavoidable in the sense that under standard complexity-theoretic assumptions, there is no algorithm that provably computes Wasserstein barycenters in polynomial time.

\paragraph*{Algorithms with exponential dependence in $d$}  A popular approach is to use ``fixed-support approximations''; that is, assume that the barycenter is supported on  a guessed set $S \subset \R^d$ of points, and then optimize over the corresponding weights, see, e.g.,~\citep{CutDou14,BenCarCut15,solomon2015convolutional,CarObeOud15,staib2017parallel,KroDviDvuetal19,janati2020debiased,lin2020fixed} among many others. The point of this fixed-support approximation is that it reduces the barycenter problem to a polynomial-size LP---which can then be solved efficiently using out-of-the-box LP solvers or specially-tailored approaches such as entropic regularization---if the set $S$ has polynomial size. However, this ``if'' is the key issue: obtaining a barycenter that is $\eps$-additively approximate for the objective~\eqref{eq:bary} requires taking $S$ to be an $\eps$-cover of the space. In particular, this means that all fixed-support methods require $\Omega((R/\eps)^{d})$ time. Such running times have two issues. First is the exponential scaling in the dimension $d$. Second is that they only compute to ``low precision'' $\eps$ due to the $1/\eps$ dependence. While not fixed-support approaches, the Frank-Wolfe algorithm of~\citep{luise2019sinkhorn} and the Functional Gradient Descent algorithm of~\citep{shen2020sinkhorn} also suffer from the same two issues.

\par Recent work has shown that in any fixed dimension $d$, Wasserstein barycenters can in fact be computed exactly in $\poly(n,k,\log U)$ time~\citep{AltBoi20bary}. However, the running time dependence on dimension is still exponential: for non-constant $d$, the running time is $(nk)^d \cdot \poly(n,k,\log U)$. Theorem~\ref{thm:np} of this paper shows that this is optimal in the sense that unless $\P = \NP$, the exponential dependence on $d$ cannot be improved to polynomial.

\paragraph*{Algorithms with exponential dependence in $k$} A well-known approach that avoids exponential dependence on the dimension $d$ is to reformulate the Wasserstein barycenter as a linear program (LP) and then solve it. However, this LP has $n^k$ variables (see, e.g.,~\citep{BenCarCut15,anderes2016discrete}), so applying a standard LP solver out-of-the-box requires $\Omega(n^k)$ time which is exponential in $k$.

\paragraph*{$2$-approximation}~\citep{Bor17} proposes the following algorithm: fix the support of $\nu$ to be the union of the supports of the input measures $\mu_i$, and optimize the corresponding $nk$ weights via an LP solver.~\citep{Bor17} shows that this yields a multiplicative $2$-approximation to the optimal barycenter problem~\eqref{eq:bary} in $\poly(n,k,d,\log U)$ time, and that this approximation factor is tight (i.e., there exist inputs for which this algorithm yields objective exactly twice the optimal). This is the polynomial-time algorithm with the best provable approximation guarantees we are aware of for the barycenter problem in high dimensions. In fact, our Theorem~\ref{thm:inapprox} implies that this polynomial-time algorithm is nearly optimal in the sense that this multiplicative $2$-approximation factor is unimprovable to a $(1\plus \eps)$-approximation under standard complexity theory assumptions.

\subsubsection{Hardness of the sparsest Wasserstein barycenter} Perhaps the most related $\NP$-hardness result is that finding the sparsest\footnote{In~\citep[Theorem 3]{BorPat19}, the $\NP$-hardness is stated for the problem of finding a Wasserstein barycenter with sparsity at most some input integer $N$. This is polynomial-time equivalent to the problem of finding the barycenter with smallest sparsity. Indeed an answer to the latter problem is an answer to the former, and an algorithm for the former problem gives an answer to the latter by running the algorithm on all $N \leq nk-k+1$ (since there always exists a barycenter with sparsity $nk-k+1$~\citep{anderes2016discrete}).} Wasserstein barycenter is $\NP$-hard, even in the setting of $k=3$ uniform measures in dimension $d=2$~\citep{BorPat19}. The key difference from the results in the present paper is that the results of~\citep{BorPat19} apply to the problem of finding the \emph{sparsest} barycenter, and do not imply $\NP$-hardness of finding a barycenter with sparsity that is polynomial in the input size, which is typically the goal in applications. For example, for the setting of $k=3$ measures, while the result of~\citep{BorPat19} shows $\NP$-hardness of finding a barycenter with sparsity $n$, a barycenter with sparsity $O(n)$ can be found in $\poly(n, \log U)$ time by using off-the-shelf LP solvers on the MOT formulation of the Wasserstein barycenter problem~\citep{anderes2016discrete,BenCarCut15}. Similarly, for any fixed dimension $d \geq 2$, while the result of~\citep{BorPat19} shows $\NP$-hardness of finding a barycenter with sparsity $n$, a barycenter with sparsity $O(nk)$ can be found in $\poly(n,k, \log U)$ time for arbitrary $k$~\citep{AltBoi20bary}.

\subsubsection{Multimarginal Optimal Transport} It is well-known that the (generalized) Wasserstein barycenter problem is equivalent to a Multimarginal Optimal Transport ($\MOT$) problem with a particular cost tensor $C$~\citep{carlier2010matching,anderes2016discrete,BenCarCut15,AguCar11}, details recalled in the preliminaries section \S\ref{ssec:prelim:mot}. Briefly, $\MOT$ is an exponential-size LP in the sense that it is an LP with $n^k$ variables. Since this is exponentially large in the input size of the barycenter problem, applying LP solvers out-of-the-box takes $\Omega(n^k)$ time which is not polynomial in $n$ and $k$. However, it is important to emphasize that the fact that this $\MOT$ problem has exponentially many variables does \emph{not} in itself imply that it cannot be solved in polynomial time. Whether an $\MOT$ problem can be solved efficiently depends on the cost $C$; indeed, a recent line of work has shown that for certain ``structured'' cost tensors $C$, the corresponding $\MOT$ problems can be solved in time that is polynomial in $n$ and $k$~\citep{AltBoi20motalg,BenCarCut15,CarObeOud15,haasler2020multi,haasler2021graphical,elvander2020multi,nenna2016numerical,benamou2019generalized,benamou2016numerical,benamou2019entropy}.

\par An obvious first requirement for an $\MOT$ problem to be solvable in polynomial time is that the cost tensor $C$ is input implicitly, since if $C$ is input explicitly then even reading the input takes $n^k$ time since $C$ has $n^k$ entries. The $\MOT$ cost $C$ corresponding to the barycenter problem satisfies this: it can be input implicitly since each entry of $C$ can be computed efficiently on-the-fly, see \S\ref{ssec:prelim:mot}. However, it is important to emphasize that just because a cost tensor $C$ has a concise implicit representation does \emph{not} imply that the corresponding $\MOT$ problem can be solved in $\poly(n,k)$ time. (See~\citep{AltBoi20mothard} for $\NP$-hard examples.)

\par This has motivated a systematic investigation into what structure makes $\MOT$ tractable. Recent work has identified a necessary~\citep{AltBoi20mothard} and sufficient~\citep{AltBoi20motalg} condition for an $\MOT$ problem to be solvable in $\poly(n,k)$-time solvable: namely, an auxiliary discrete optimization problem depending on $C$ must also be solvable in polynomial time. These two papers respectively use this result to show that for certain commonly arising families of cost tensors, the corresponding $\MOT$ problems are either $\NP$-hard or polynomial-time solvable. 
However, the particular cost $C$ corresponding to the barycenter problem does not fall under any previous hardness results for $\MOT$, and thus we require new techniques.

\subsubsection{Other related work}

\paragraph*{Algorithms based on entropic regularization} The influential paper~\citep{cuturi2013sinkhorn} popularized the use of entropic regularization for large-scale Optimal Transport computation. The use of entropic regularization to compute Wasserstein barycenters was first proposed in~\citep{CutDou14}, which inspired a long line of work, see, e.g.,~\citep{BenCarCut15,solomon2015convolutional,KroDviDvuetal19,janati2020debiased,lin2020fixed,luise2019sinkhorn,shen2020sinkhorn}. Intuitively, the idea is to regularize the resulting LP by adding $\delta$ times an entropy cost, for $\delta$ small. 
This makes the LP strongly convex and easier to optimize. Previous work has sought to design barycenter algorithms by judiciously choosing $\delta$ and designing specialized algorithms for the resulting $\delta$-regularized barycenter problem.
An immediate corollary of our main results (Theorems~\ref{thm:inapprox} and~\ref{thm:inapproxext}) is that entropic regularization does not help for computing barycenters in high dimensions: under standard complexity assumptions, there is no efficient algorithm for the (generalized) Wasserstein barycenter problem regardless of whether one uses entropic regularization.

\paragraph*{Continuous distributions} While this paper and much of the literature focuses on computing Wasserstein barycenters of discrete distributions, there is also an interesting line of work on computing barycenters of continous distributions. This continuous setting has several additional computational challenges, such as how to even represent $\mu_i$ and $\nu$ concisely, and how to compute the Wasserstein distance between them efficiently. Due to these computational issues, the literature on barycenters of continuous distributions typically restricts to Gaussians, in which case specialized algorithms can be designed; see, e.g.,~\citep{chewi2020gradient,AlvEtAl16}.

\subsection{Outline}

In \S\ref{sec:prelim} we establish preliminaries and notation. We prove our main results in \S\ref{sec:hard} and \S\ref{sec:ext}. Specifically, in \S\ref{sec:hard} we prove hardness of computation for the standard Wasserstein barycenter problem (Theorems~\ref{thm:np} and~\ref{thm:inapprox}), and in \S\ref{sec:ext} we prove hardness of computation for the generalized Wasserstein barycenter problem (Theorem~\ref{thm:inapproxext}). While the former is implicit from the latter, we provide this separation for expository purposes since the proof for the standard barycenter problem is less involved. In \S\ref{sec:conc} we conclude with future research directions.

\section{Preliminaries}\label{sec:prelim}

\subsection{Notation}\label{ssec:prelim:notation}

\paragraph*{Barycenters} The optimal value of the generalized barycenter problem~\eqref{eq:baryext} is denoted by $\optbary$. We show that the claimed hardness results hold even in the special case where the weights $\lambda_1 = \dots = \lambda_k = 1/k$ are uniform, 
and thus henceforth specialize solely to this case. The atoms in the support of distribution $\mu_i$ are denoted by $x_{i,1}, \dots, x_{i,n} \in \R^d$. We abuse notation slightly by writing $\mu_i$ to denote this discrete distribution as well as the vector of probability masses in the simplex $\Delta_n = \{p \in \R_{\geq 0}^n : \sum_{i=1}^n p_i = 1\}$ over the $n$ atoms $\{x_{i,j}\}_{j=1}^n$ in any fixed ordering.
The Euclidean norm is denoted by $\|\cdot\|$, and the dot product is denoted by $\langle \cdot, \cdot \rangle$. For shorthand, we often write $\sum_{i < i' \in [k]}$ to denote the sum over pairs $(i,i')$ satisfying $1 \leq i < i' \leq k$.

\paragraph*{Tensors} We denote the $k$-fold product space $\Rn \otimes \cdots \otimes \Rn$ by $\Rntk$, and similarly for $\Rpntk$. The $i$-th marginal of a tensor $P \in \Rntk$ is the vector $m_i(P) \in \Rn$ with $j$-th entry $[m_i(P)]_j = \sum_{j_1,\ldots,j_{i-1},j_{i+1},\ldots,j_k} P_{j_1,\ldots,j_{i-1},j,j_{i+1},\ldots,j_k}$.
The set $\{1, \dots, n\}$ is denoted by $[n]$, and the $k$-fold product space $[n] \times \cdots \times [n]$ is denoted by $[n]^k$. For shorthand, we often denote an element of $[n]^k$ by $\jvec = (j_1,\dots,j_k)$. We denote the maximum modulus entry of a tensor $C$ by $\Cmax = \max_{\jvec} |C_{\jvec}|$, and the inner product of two tensors $A,B \in \Rntk$ by $\langle A, B \rangle = \sum_{\jvec} A_{\jvec}B_{\jvec}$.

\paragraph*{Bit complexity} For simplicity, we ignore discussion of bit complexity throughout since the constructed ``hard'' instances of (generalized) Wasserstein barycenters are such that all support points $x_{i,j} \in \R^d$ have $\{-1,0,1\}$-valued entries, and thus clearly have polynomial bit complexity. 

\paragraph*{Complexity theory}~We recall the definition of the complexity class $\BPP$, which appears in the statement of Theorems~\ref{thm:inapprox} and \ref{thm:inapproxext}. A language $L \subset \{0,1\}^*$ is in $\BPP$ if there exists a polynomial-time randomized Turing Machine $M$ such that for every $x \in \{0,1\}^*$, the machine $M$ decides whether $x$ is in the language with error probability at most $1/3$: i.e., $\PP [M(x) = \mathds{1}(x \in L)] \geq 2/3$, where the probability is over the internal randomness used by $M$. Under standard cryptographic assumptions, it is known that $\NP \not\subset \BPP$; see, e.g., Chapter 20 of \citep{arora2009computational}.

\subsection{Multimarginal Optimal Transport formulation}\label{ssec:prelim:mot}

We make use of the well-known fact that the (generalized) Wasserstein barycenter problem has an equivalent formulation as a certain exponential-size linear program, namely a certain Multimarginal Optimal Transport ($\MOT$) problem~\citep{carlier2010matching,anderes2016discrete,BenCarCut15,AguCar11}. We recall the details of this formulation here. 

\par $\MOT$ is the problem of linear progamming over joint probability distributions with fixed marginals. More precisely, given measures $\mu_1, \dots, \mu_k \in \Delta_n$ and a cost tensor $C \in \Rntk$, the corresponding $\MOT$ problem is
\begin{align}
	\min_{P \in \Coup} \langle C, P \rangle.
	\label{eq:MOT}
\end{align}
Above, $\Coup$ denotes the transportation polytope $\{ P \in \Rpntk : m_i(P) = \mu_i, \; \forall i \in [k]\}$, a well-studied object in the optimization and combinatorics communities, see, e.g.,~\citep{de2014combinatorics}.

\begin{prop}[$\MOT$ formulation]\label{prop:mot} 
	Suppose $p \in [1,\infty)$ and $q \in [1,\infty]$. The value of the generalized barycenter problem~\eqref{eq:baryext} for measures $\mu_1, \dots, \mu_k$ is equal to the value of the $\MOT$ problem~\eqref{eq:MOT} with marginals $\mu_1, \dots, \mu_k$ and cost tensor $C \in \Rntk$ that has entries
	\begin{equation}
		C_{\jvec} = \min_{y \in \RR^d}  \sum_{i=1}^k \lambda_i \|x_{i,j_i} - y\|_q^p.\label{eq:barycostextension}
	\end{equation}
\end{prop}

Although details of this proposition's proof are not necessary for the rest of the paper, we briefly recall the main idea in order to provide the reader intuition behind the connections between these two problems. In fact, more is true about this connection: not only are the optimal values of these two problems equal (Proposition~\ref{prop:mot}), but also their solutions are in ``correspondence'' in the sense that, given a solution to either one of these two problems, one can construct a solution to the other problem with the same objective value. One direction of this transformation is particularly simple to explain: if $P \in \Coup$ is an optimal solution to this $\MOT$ problem, then the pushforward of $P$ under the map\footnote{Although this does not necessarily define a map when $p=1$ (e.g., since the geometric median of points in $\R^d$ is not necessarily unique), taking any minimizer $y$ suffices.} $(X_1, \dots, X_k) \mapsto \argmin_{y \in \R^d} \sum_{i=1}^k \lambda_i \|X_i - y\|_q^p$ is an optimal generalized barycenter with the same value. The other direction intuitively inverts this transformation, but is somewhat more involved to state precisely; for brevity, we refer the reader to~\citep[Section 6]{carlier2010matching} for a formal proof.

\subsection{Toolbox for analyzing the complexity of specific Multimarginal Optimal Transport problems}\label{ssec:prelim:min}

Here, we recall the recent results of \citep{AltBoi20mothard} that reduce (approximately) computing the minimum entry $\minprob$ of a tensor $C$, to (approximately) computing the optimal value of the $\MOT$ problem with cost $C$. We emphasize that this applies to an arbitrary cost tensor $C \in \Rntk$; not just those of the form~\eqref{eq:barycostextension} corresponding to the (generalized) Wasserstein barycenter problem.
The benefit of this reduction is that $\minprob$ is a combinatorial optimization problem that is phrased in a more amenable way for proving NP-hardness. 
Below, for a cost tensor $C \in \Rntk$, let $\MOT_C$ denote the problem of computing the optimal value~\eqref{eq:MOT} of $\MOT$ with cost $C$, given marginal distributions $\mu_1, \dots, \mu_k \in \Delta_n$.

\begin{prop}[Simplified version of Theorem 3.1 of~\citep{AltBoi20mothard}]\label{prop:min}
	There is a deterministic algorithm that, given access to an oracle solving $\MOT_C$, computes $\minprob$ in $\poly(n,k)$ oracle queries and additional time.
\end{prop}

\begin{prop}[Simplified version of Theorem 3.2 of~\citep{AltBoi20mothard}]\label{prop:amin}
	There exists a constant $\alpha \in \N$ and a randomized algorithm that, given $\eps > 0$ and access to an oracle solving $\MOT_C$ to additive accuracy $\eps$, computes $\minprob$ up to $\eps (nk)^{\alpha}$ additive accuracy with probability $2/3$ in $\poly(n,k,\Cmax/\eps)$ oracle queries and additional time.
\end{prop}

Although details of these propositions' proofs are not necessary for the rest of the paper, we briefly recall the main ideas in order to make the paper more self-contained. The key lemma is that the discrete optimization problem $\minprob$ admits an exact convex relaxation with the following property: evaluating the objective function of this convex optimization problem amounts to solving an $\MOT_C$ problem. Exactness ensures that solving this convex optimization problem suffices to solve $\minprob$. The evaluation property ensures that one can solve the convex optimization problem using zero-th order optimization algorithms, where in each iteration the objective function is evaluated by solving an auxiliary $\MOT_C$ problem. By appealing to standard results about zero-th order optimization algorithms and analyzing the smoothness properties of the relevant convex optimization problem, it follows that polynomially many iterations suffice, and thus only polynomially many $\MOT_C$ computations suffice. If the $\MOT_C$ oracle used for the function evaluations are exact (resp., approximate), then the computed solution to $\minprob$ is exact (resp., approximate).

\subsection{Clique}\label{ssec:prelim:clique}

For a graph $G = (V,E)$ and vertices $v_1, \dots, v_k \in V$, we denote the number of edges in the induced subgraph of $G$ with these vertices by $|E(v_1, \dots, v_k)| = \sum_{i < i'\in [k]} \mathds{1}[(v_i,v_{i'}) \in E]$. In the sequel, it is convenient to consider this quantity $|E(v_1, \dots, v_k)|$ in the general case where the vertices are not necessarily distinct; the definition extends as written to this case, and counts the number of edges with multiplicity. 

\par We make use of the well-known fact that the $\clique$ decision problem is $\NP$-hard~\citep{karp1972reducibility}. Recall that a $k$-clique in an undirected graph $G = (V,E)$ is a subset of vertices $S \subseteq V$ of size $|S| = k$ such that $|E(S)| = \binom{k}{2}$. The $\clique$ decision problem is: given an undirected graph $G = (V, E)$ and an integer $k > 0$, decide whether $G$ contains a $k$-clique. For technical reasons, it is convenient to use the fact that this $\clique$ problem is $\NP$-hard even in certain special cases; a proof is provided in Appendix~\ref{supp:clique}.

\begin{prop}[Hardness of $\clique$]\label{prop:clique-special}
	Assuming $\P \neq \NP$, there does not exist an algorithm that, given an integer $k > 0$ and a graph $G$ on $n$ vertices, decides whether $G$ contains a clique of size $k$ in $\poly(n,k)$ time. This is unchanged if $G$ is assumed regular and $k$ is even.
\end{prop}

\section{Case of $p=q=2$: standard Wasserstein barycenters}\label{sec:hard}

In this section, we establish the hardness of the standard Wasserstein barycenter problem ($p=q=2$) for exact and approximate computation by proving Theorems~\ref{thm:np} and \ref{thm:inapprox}, respectively. We note that although these results follow from the general setting of $p \in [1,\infty)$ and $q \in [1,\infty]$ studied in Section~\ref{sec:ext}, for expository purposes we isolate here the proof for the case of $p=q=2$ because the general case is significantly more involved. We refer the reader to the techniques section \S\ref{ssec:intro:tech} for a high-level proof overview. 

Our starting point is that by combining the two known reductions recalled in the preliminaries section, showing hardness of (approximately) computing the Wasserstein barycenter problem reduces to showing hardness of (approximately) computing the $\chub_{2,2}$ problem. Establishing the latter hardness statement is therefore the purpose of this section. Note that the $\NP$-hardness result we show for $\chub_{2,2}$ in general dimension $d$ stands in starks contrast to the algorithm of~\citep{AltBoi20bary} which solves $\chub_{2,2}$ exactly in $\poly(n,k)$ time for any fixed dimension $d$.

\begin{lemma}[Inapproximability of $\chub_{2,2}$]\label{lem:minoraclenphard}	
	Given vectors $\{x_{i,j}\}_{i \in [k], j \in [n]} \subseteq \{0,1\}^{d}$, it is $\NP$-hard to compute the value of $\chub_{2,2}$ (see Definition~\ref{def:chub}) to additive error $0.99/k$.
\end{lemma}

\par Following, we make precise how Lemma~\ref{lem:minoraclenphard} implies Theorems~\ref{thm:np} and \ref{thm:inapprox}. The rest of this section is then devoted to proving this new key Lemma~\ref{lem:minoraclenphard}.

\begin{proof}[Proof of Theorem~\ref{thm:np}]
	If there is a $\poly(n,k,d)$-time algorithm for computing $\optbary$, then by Proposition~\ref{prop:mot} there is a $\poly(n,k,d)$-time algorithm for computing $\MOT_C$ with the cost tensor $C$ given by~\eqref{eq:barycostextension}, thus by Proposition~\ref{prop:min} there is a $\poly(n,k,d)$-time algorithm for $\chub_{2,2}$. Assuming $\P \neq \NP$, this contradicts the $\NP$-hardness of $\chub_{2,2}$ in Lemma~\ref{lem:minoraclenphard}. 
\end{proof}

\begin{proof}[Proof of Theorem~\ref{thm:inapprox}]
	If there is a $\poly(n,k,d,R/\eps)$-time randomized algorithm for $\eps$-approximating $\optbary$ with probability $2/3$, then by Proposition~\ref{prop:mot} there is a $\poly(n,k,d,R/\eps)$-time randomized algorithm for $\eps$-approximating $\MOT_C$ with probability $2/3$. By a standard boosting argument---namely repeating this algorithm $\log 1/\delta$ times, taking the median, and applying a Chernoff bound---this implies a $\poly(n,k,d,R/\eps, \log 1/ \delta)$-time randomized algorithm for $\eps$-approximating $\MOT_C$ with probability $1-\delta$. Therefore, by setting $\delta$ sufficiently high, we conclude by Proposition~\ref{prop:amin} and a union bound
	that there exists a $\poly(n,k,d,R/\eps)$-time randomized algorithm for $\eps (nk)^{\alpha}$-approximating $\chub_{2,2}$ with probability of success at least $0.51$. This can be boosted to $2/3$ probability of success by another standard boosting argument, proving that $\chub_{2,2}$ lies in $\BPP$.  However, assuming $\NP \not\subset \BPP$, this contradicts the $\NP$-hardness in Lemma~\ref{lem:minoraclenphard}.
\end{proof}

\subsection{Proof of Lemma~\ref{lem:minoraclenphard}}\label{ssec:bary:proof}

As discussed in the Techniques section~\ref{ssec:intro:tech}, the proof relies on appropriately embedding the adjacency properties of a graph $G$ into a point configuration in order
to encode the $\NP$-hard $\clique$ decision problem as an instance of $\chub_{2,2}$ in sufficiently high dimension $d$. This embedding is as follows; see Remark~\ref{rem:embed} below for an interpretation of it as the edge-indicator vector embedding of a certain augmented graph. For simplicity of exposition, we make no attempt to optimize $d$ here (dimensionality reduction can be done, e.g., by simply applying the Johnson-Lindenstrauss lemma).

\begin{lemma}[Embedding for $\clique$]\label{lem:graphembedding}
	Given an $n$-vertex $D$-regular graph $G = (V,E)$ and an integer $k > 0$, there exists a function $\phi : [k] \times [n] \to \{0,1\}^{\binom{k}{2}n^2}$ satisfying the following.
	\begin{itemize}
		\item[(i)] $\phi$ can be evaluated in $\poly(n,k)$ time.
		\item[(ii)] For all $i \neq i' \in [k]$ and $v,v' \in V$, it holds that $\langle \phi(i,v), \phi(i',v') \rangle = \mathds{1}[(v,v') \in E]$.
		\item[(iii)] For all $i \in [k]$ and $v \in V$, it holds that $			\|\phi(i,v)\|_2^2 = D(k-1)$.
	\end{itemize}
\end{lemma}
\begin{proof}
	Define the embedding $\phi$ as follows. Index the $\binom{k}{2}n^2$ coordinates by tuples $(\ell,u,\ell',u')$ for indices $\ell < \ell' \in [k]$ and $u,u' \in [n] \cong V$. Set 
	\begin{align}
		\left[\phi(i,v)\right]_{(\ell,u,\ell',u')}
		=
		\mathds{1}[(u,u') \in E] \cdot \mathds{1}[(\ell,u) = (i,v)\mbox{ or } (\ell',u') = (i,v)].
		\label{eq:embedding-main}
	\end{align}
	Property (i) clearly holds since each entry of $\phi$ is computable in $\poly(n,k)$ time. 
	
	\par To show property (ii), note that the vectors $\phi(i,v)$ and $\phi(i',v')$ have disjoint support if $(v,v') \notin E$, and otherwise share exactly one non-zero entry at the coordinate $(\ell,u,\ell',u') = (i,v,i',v')$. Thus $\langle \phi(i,v), \phi(i',v') \rangle = \mathds{1}[(v,v') \in E]$.
	
	\par To show property (iii), note that the squared norm $\|\phi(i,v)\|_2^2$ is equal to the number of non-zeros in the embedded vector $\phi(i,v)$ since it is an indicator vector. Since $G$ is $D$-regular, counting the number of non-zero entries in $\phi(i,v)$ shows that $\|\phi(i,v)\|_2^2 = D(k-1)$.
\end{proof}

\begin{remark}[Interpretation of embedding via tensor-product graph]\label{rem:embed}
	The embedding $\phi : [k] \times [n] \to \R^d$ is the edge-indicator embedding of the graph $\tilde{G}$ that is the tensor product of the complete graph on $k$ vertices and $G$. That is, $\tilde{G}$ is $k$-partite and has vertex set $\tilde{V}$ equal to $k$ independent copies of $V$. A vertex in $\tilde{V}$ can be indexed by a tuple $(i,v)$, where $i \in [k]$ denotes the copy index and $v \in V$ denotes the corresponding vertex in the original graph. Two vertices $(i,v)$ and $(i',v')$ are adjacent in $\tilde{G}$ if and only if $i \neq i'$ and $(v,v') \in E$; that is, if and only if these two vertices in $\tilde{V}$ are from different copies of $V$ and are such that their underlying vertex indices are adjacent in the original graph $G$.
\end{remark}

\begin{proof}[Proof of Lemma~\ref{lem:minoraclenphard}]
	We reduce $\clique$ to approximately solving $\chub_{2,2}$. Given an $n$-vertex $D$-regular graph $G = (V, E)$ and an integer $k > 0$, let $\phi$ be the corresponding embedding in Lemma~\ref{lem:graphembedding}. Set
	$
	x_{i,j} = \phi(i,j) \in \{0,1\}^{d}
	$
	for each $i \in [k]$ and $j \in [n] \cong V$, where $d = \binom{k}{2}n^2$.
	\par Recall that the $\chub_{2,2}$ for the input points $\{x_{i,j}\}_{i \in [k], j \in [n]} \subset \R^d$ is the problem of computing the value
	\begin{align}
		F^* := \min_{(j_1,\dots,j_k) \in [n]^k} \min_{y \in \R^d} \sum_{i=1}^k \|x_{i,j_i} - y\|_2^2,
		\label{eq:chub-1}
	\end{align}
	see Definition~\ref{def:chub}. This objective simplifies for the particular choice of input points. Indeed,
	\begin{align}
		\min_{y \in \R^d} \sum_{i=1}^k \|x_{i,j_i} - y\|_2^2
		&=
		\left( 1 - \frac{1}{k} \right) \sum_{i=1}^k \|x_{i,j_i}\|_2^2 
		- \frac{2}{k} \sum_{i < i' \in [k]} \langle x_{i,j_i}, x_{i',j_i'} \rangle \nonumber
		\\
		&=
		D(k-1)^2
		- \frac{2}{k} \sum_{i < i' \in [k]}  \mathds{1}[(j_i,j_{i'}) \in E]
		\nonumber
		\\ &= M - \frac{2}{k} |E(j_1,\dots,j_k)|.
		\label{eq:chub-2}
	\end{align}
	Above, the first step is by plugging in the closed-form solution for $y = \frac{1}{k}\sum_{i=1}^k x_{i,j_i}$. The second step is by using the key properties (ii) and (iii) of the embedding $\phi$ in Lemma~\ref{lem:graphembedding}. The third step is by defining $M := D(k-1)^2$ and recalling that we write $|E(j_1,\dots,j_k)| = \sum_{i < i' \in [k]} \mathds{1}[(j_i,j_{i'}) \in E]$ to denote the number of edges in the induced subgraph of $G$ with vertices $j_1,\dots,j_k$ where we count the edges with multiplicity if $j_1, \dots, j_k$ are not distinct (see \S\ref{ssec:prelim:notation}).
	
	\par Therefore by combining~\eqref{eq:chub-1} and~\eqref{eq:chub-2}, we conclude that the value $F^*$ of the $\chub_{2,2}$ problem for the particular chosen input points $\{x_{i,j}\}_{i \in [k], j \in [n]} \subset \R^d$ is equal to
	\begin{align}
		F^*
		=
		M - \frac{2}{k} \cdot \max_{(j_1,\dots,j_k) \in [n]^k} |E(j_1,\dots,j_k)|
		=
		M - \frac{2}{k} \cdot \max_{S \in V^k} |E(S)|.
		\label{eq:chub-3}
	\end{align}
	Note that in this final equation, the optimization is over a multiset $S$ of $k$ vertices in $V$ that are not necessarily distinct. This is a multiset rather than a set because $\chub_{2,2}$ optimizes over tuples $(j_1,\dots,j_k) \in [n]^k$, and such a tuple does not necessarily consist of distinct indices; this is also why we defined $|E(S)|$ to count edges with multiplicity. 
	\par Using~\eqref{eq:chub-3}, we now argue that the value $F^*$ of $\chub_{2,2}$ varies significantly depending on whether $G$ contains a clique of size $k$. Specifically, on one hand 
	\begin{align}
		F^* = M-k+1, \qquad \text{ if $G$ contains a clique of size $k$ },
		\label{eq:chub-4}
	\end{align}
	because using this clique as the set $S$ and plugging into~\eqref{eq:chub-3} yields objective value $M - \frac{2}{k} \cdot \binom{k}{2} = M - k + 1$. On the other hand, 
	\begin{align}
		F^* \geq M-k+1 - \frac{2}{k}, \qquad \text{ if $G$ does not contain a clique of size $k$ },
		\label{eq:chub-5}
	\end{align}
	because then $\max_{S \in V^k} |E(S)| \leq \binom{k}{2} - 1$, whereby from~\eqref{eq:chub-3} we conclude that $F^* \geq M - \frac{2}{k}(\binom{k}{2} - 1) = M - k+1 - \frac{2}{k}$. It therefore follows from~\eqref{eq:chub-4} and~\eqref{eq:chub-5} that computing $\chub_{2,2}$ to any additive error less than $\frac{1}{k}$ enables one to decide whether $G$ contains a clique of size $k$, which is an $\NP$-hard problem~\citep{karp1972reducibility}.
\end{proof}

\section{Hardness of generalized Wasserstein barycenters}\label{sec:ext}

In this section, we prove Theorem~\ref{thm:inapproxext}. This extends our hardness-of-computation results to averaging probability distributions with respect to other Optimal Transport metrics by varying the $q$ parameter, and also different notions of Fr\'echet $p$-means by varying the $p$ parameter.

\begin{lemma}[Generalization of~Lemma~\ref{lem:minoraclenphard}]\label{lem:minexthard}
	Fix $p \in [1, \infty)$ and $q \in [1, \infty]$. There is a constant $\alpha = \alpha(p,q) < \infty$ such that, given vectors $\{x_{i,j}\}_{i \in [k], j \in [n]} \subseteq \{-1,0,1\}^{d}$, it is $\NP$-hard to compute the value of $\chubpq$ to additive error $\eps = (nk)^{-\alpha}$.	
\end{lemma}

The proof of Theorem~\ref{thm:inapproxext} from Lemma~\ref{lem:minexthard} is identical to the proof of Theorem~\ref{thm:inapprox} from Lemma~\ref{lem:minoraclenphard} and thus is omitted for brevity. 

\par It therefore remains to show Lemma~\ref{lem:minexthard}. As detailed in \S\ref{ssec:intro:tech}, our high-level approach is similar to the proof for the case of $p=q=2$ in \S\ref{sec:hard} in that we reduce $k$-clique to $\chubpq$. However, for general $p$ and $q$, the proof is significantly more involved in large part because the objective function $F_{p,q}$ (see~\eqref{eq:F-def}) does not have a closed-form solution. Our argument is based on the following key lemma.

\begin{lemma}[Gap between cliques and non-cliques]\label{lem:gap}
	Given an $n$-vertex, $D$-regular graph $G$, an even integer $k > 0$, and parameters $p \in [1,\infty),q \in [1,\infty]$, there is an algorithm that takes $\poly(n,k)$ time to compute vectors $\{x_{i,j}\}_{i \in [k], j \in [n]} \in \{-1,0,1\}^d$ satisfying the following. 
	\begin{itemize}
		\item If $v_1, \dots, v_k \in V$ form a $k$-clique in $G$, then
		$$\val_{p,q}(x_{1,v_1},\ldots,x_{k,v_k}) \leq \gamma.$$
		\item If $v_1, \dots, v_k \in V$ does not form a $k$-clique in $G$, then
		$$\val_{p,q}(x_{1,v_1},\ldots,x_{k,v_k}) \geq \gamma + \Delta.$$
	\end{itemize}
	Here, $\gamma := \gamma(n,k,D,p,q)$ and $\Delta := \Delta(n,k,D,p,q)$ are quantities that can be computed in $\poly(n,k)$ time for any fixed $p \in [1,\infty), q \in [1,\infty]$. Furthermore, $\Delta \geq (nk)^{-\alpha}$ for some $\alpha := \alpha(p,q)$ that depends only on $p,q$.
\end{lemma}

Note that in this lemma, the claim holds even when the vertices $v_1, \dots, v_k$ are not distinct.
\par Given Lemma~\ref{lem:gap}, we now prove Lemma~\ref{lem:minexthard}.

\begin{proof}[Proof of Lemma~\ref{lem:minexthard}]
	Suppose for sake of contradiction that there is an algorithm that given points $\{x_{i,j}\}_{i \in [k], j \in [n]} \in \{-1,0,1\}^d$ and accuracy $\eps > 0$, computes $F^* := \min_{\vec{j} \in [n]^k} \val_{p,q}(x_{1,j_1},\ldots,x_{k,j_k})$ to $\eps$ additive error in $\poly(n,k,d,1/\eps)$ time. Then, given any $n$-vertex, $D$-regular graph $G$, and an even integer $k > 0$, we give a $\poly(n,k)$-time algorithm that determines whether $G$ contains a $k$-clique. This contradicts the $\NP$-hardness of $\clique$ given in Proposition~\ref{prop:clique-special}, and proves the lemma.
	
	First, compute the graph embedding vectors $\{x_{i,j}\}_{i \in [k], j \in [n]} \in \RR^d$, the threshold $\gamma = \gamma(n,k,D,p,q)$ and the gap $\Delta = \Delta(n,k,D,p,q)$ in $\poly(n,k)$ time using Lemma~\ref{lem:gap}. Second, compute $F^*$ up to additive error $\Delta /3$, which takes $\poly(n,k)$ time by our assumption, since $d, \Delta^{-1} \leq \poly(n,k)$. This approximate value lets us distinguish the case in which $G$ has a $k$-clique and $F^* \leq \gamma / k$ from the case in which $G$ does not have a $k$-clique and $\min_{\vec{j}} C_{\vec{j}} \geq \gamma / k + \Delta / k$, which solves the $\clique$ problem in $\poly(n,k)$ time, contradicting its $\NP$-hardness.
\end{proof}
\par Therefore it only remains to prove Lemma~\ref{lem:gap}. The following helper lemma is useful for this.
\begin{lemma}\label{lem:pq-convexity}
	Let $p \in [1,\infty)$, $q \in [1,\infty]$, and $x \in \R^d$. The function
	$y \mapsto \|x-y\|_q^p$
	is convex on $\R^d$.
\end{lemma}
\begin{proof}
	Express this function as the composition $f \circ g \circ h$ of the powering function $f(t) = t^p$ from $\R \to \R$, the $q$-norm function $g(z) = \|z\|_q$ from $\R^d \to \R$, and the translation function $h(y) = x-y$ from $\R^d \to \R$, and appeal to standard results on convexity-preserving transformations (see e.g.,~\citep{Boyd04}). Specifically, since $g$ is convex and $h$ is linear, their composition $g \circ h$ is convex. Thus, since also $f$ is convex and monotonic, it follows that $f \circ (g \circ h) = f \circ g \circ h$ is convex.
\end{proof}

In Sections \ref{app:q1inf}, \ref{app:q1}, and \ref{app:qinf}, we prove Lemma~\ref{lem:gap} for the cases $q \in (1,\infty)$, $q = 1$, and $q = \infty$ respectively. These three cases can be read separately, and together prove Theorem~\ref{thm:inapproxext}.

\subsection{Case $q \in (1, \infty)$}\label{app:q1inf}
We prove Lemma~\ref{lem:gap} in the case of $q \in (1,\infty)$ and general $p \in [1,\infty)$.  As in the $p=q=2$ case proved in \S\ref{sec:hard}, our proof strategy is to reduce from the $\clique$ problem, and we use the same graph embedding $\phi$ as in \eqref{eq:embedding-main}, which we restate below for convenience. Namely, given a graph $G = (V,E)$ on vertex set $V \cong [n]$, we define the embedding $\phi = \phi_G : [k] \times [n] \to \mathbb{R}^{d}$ by letting $d = \binom{k}{2} n^2$ and indexing the $d$ coordinates by tuples $(\ell, u,\ell',u')$ where $\ell < \ell' \in [k]$ and $u,u' \in [n] \cong V$, and setting:
\begin{equation}\label{eq:embedding-q1inf}[\phi(i,v)]_{(\ell,u,\ell',u')} = \mathds{1}[(u,u') \in E] \cdot \mathds{1}[(\ell,u) = (i,v) \mbox{ or } (\ell',u') = (i,v)].\end{equation}

In words, this embedding guarantees that for any $i < i' \in [k]$ and $v,v' \in [n] \cong V$, the vectors $\phi(i,v)$ and $\phi(i',v')$ have disjoint support if $(v,v') \not\in E$, and otherwise share exactly one non-zero entry at coordinate $(\ell,u,\ell',u') = (i,v,i',v')$. Although the graph embedding is the same as the embedding used to prove the special case $p = q = 2$ in Section~\ref{sec:hard}, a new and significantly more involved analysis is required. The main challenge is that, unlike the $p = q = 2$ case, we do not have a closed-form solution to $\val_{p,q}(x_1,\ldots,x_k)$ given $x_1,\ldots,x_k \in \mathbb{R}^d$. Thus, we cannot analytically compute the value of $\val_{p,q}(\phi(1,v_1),\ldots,\phi(k,v_k))$ and compare the case of cliques and non-cliques as we did in Section~\ref{sec:hard}. This obstacle is compounded by the fact that for general $p,q$, the value of $\val_{p,q}(\phi(1,v_1),\ldots,\phi(k,v_k))$ is not even determined uniquely by the number of edges $|E(v_1,\ldots,v_k)|$ between the vertices $v_1,\ldots,v_k$.

We overcome this obstacle as follows: in order to prove that $k$-cliques give a lower cost than non-$k$-cliques, with a polynomial gap for any fixed $p$ and $q$, we argue via induction on the number of edges in the subgraph induced by the $k$ vertices that adding an edge increases the value of $F$ by at least a polynomial gap on each iteration. In particular, this allows us to conclude that if we have a $k$-clique minus any number of edges, we have a gap of $\Delta$ in $F$ versus if we have a $k$-clique.

In order to conduct this analysis, it is helpful to make the following technical definition. 

\begin{defin}\label{def:kstcollection}
	A $(k,s,t)$-collection is a collection of vectors $x_1,\ldots,x_k \in \{0,1\}^d$, for some dimension $d$, such that:
	\begin{enumerate}
		\item[(i)] Each vector is non-zero on $s$ entries, i.e., $\|x_1\|_0 = \dots = \|x_k\|_0 = s$.
		\item[(ii)] Exactly $t$ pairs of vectors share one non-zero entry, i.e., $|\{(i,i') : \langle x_i, x_{i'} \rangle = 1, i < i'\}| = t$.
		\item[(iii)] The other $\binom{k}{2} - t$ pairs of vectors are disjoint, i.e., $|\{(i,i') : \langle x_i,x_{i'}  \rangle  = 0, i < i'\}| = \binom{k}{2} - t$.
		\item[(iv)] For each entry $j \in [d]$, at most two vectors are non-zero, i.e., $|\{i : [x_i]_j \neq 0\}| \leq 2$.
	\end{enumerate}
\end{defin}

The significance of this definition is that vectors constructed using the embedding $\phi$ are $(k,s,t)$-collections, as stated below. The proof is deferred to Appendix~\ref{app:lem-whykst}.
\begin{lemma}\label{lem:whykst}
	If $G = (V,E)$ is an $n$-vertex, $D$-regular graph, and $v_1,\ldots,v_k \in [n] \cong V$ then $\phi(1,v_1),\ldots,\phi(k,v_k)$ is a $(k,D(k-1), |E(v_1,\ldots,v_k)|)$-collection.
\end{lemma}

The main benefit of writing the subsequent arguments in terms of $(k,s,t)$-collections is that we eliminate the need to check whether any given $(k,s,t)$-collection that we construct during our arguments can be instantiated as some embedding $\phi(1,v_1),\ldots,\phi(k,v_k)$ of vertices in some graph $G$. Indeed, in the subsequent arguments we construct $(k,s,t)$-collections that may not correspond to the embedding of a graph, yet our induction arguments are still meaningful, since our lemmas hold for all $(k,s,t)$-collections:

\begin{lemma}\label{lem:yilowerbound}
	Let $p \in [1, \infty)$ and $q \in (1, \infty)$. There exists a constant $\alpha > 0$ that depends only on $p$ and $q$ and satisfies the following. 
	If $x_1,\ldots,x_k \in \{0,1\}^d$ is a $(k,s,t)$-collection and at least one vector $x_i$ has a non-zero $j$-th coordinate, then any solution $y \in \RR^d$ to the optimization problem
	\[
	F_{p,q}(x_1, \dots, x_k) = \min_{y \in \R^d} \sum_{i=1}^k \|x_i - y\|_q^p
	\]
	satisfies
	\[
	y_j \geq (kd)^{-\alpha}.
	\]
\end{lemma}
\begin{proof}
	After possibly permuting the $k$ vectors, we may assume without loss of generality that $[x_i]_j = \mathds{1}[i \leq m]$ for all $i \in [k]$ and some $m \in \{1,2\}$. Note also that the solution $y \in [0,1]^d$, or else projecting $y$ to $[0,1]^d$ would improve the objective. Thus the solution $y$ satisfies
	$$\val_{p,q}(x_1,\ldots,x_k) = \sum_{i=1}^m ((1 - y_j)^q + c_i)^{\frac{p}{q}} + \sum_{i=m+1}^k (y_j^q + c_i)^{\frac{p}{q}},$$
	where $c_i = \sum_{\ell \neq j}^d |[x_i]_{\ell} - y_{\ell}|^q \in [0,d-1]$ for $i \in [k]$. By the first-order optimality condition on $y_j$,
	\[
	(1-y_j)^{q-1} \sum_{i=1}^m \left( (1-y_j)^q + c_i \right)^{p/q-1}
	=
	y_j^{q-1} \sum_{i=m+1}^k \left( y_j^q + c_i \right)^{p/q -1}.
	\]
	Re-arranging yields $y_j = a/(a+b)$, where
	$$a = \left(\sum_{i=1}^m ((1-y_j)^q + c_i)^{(p/q - 1)}\right)^{1/(q-1)} \quad \text{ and } \quad b = \left(\sum_{i=m+1}^k (y_j^q + c_i)^{(p/q - 1)}\right)^{1/(q-1)}.$$
	Assume without loss of generality that $k \geq 2$, as the $k = 1$ case is trivial (if $k = 1$ then $y_j = [x_1]_j = 1$). Thus, we may also assume without loss of generality that $y_j \leq 1/2$; the other case $y_j \geq 1/2$ is handled by adjusting the constant $\alpha$. Then $a \geq 2^{(q-p)/(q-1)}$, and $b \leq (k d^{(p/q-1)})^{1/(q-1)}$ since $y_j^q + c_i \leq d$. The claim follows by suitably choosing the constant $\alpha$.
\end{proof}

\begin{lemma}\label{lem:pq-helper}
	Let $0 \leq t' \leq t \leq T$ for $T \geq 1$.
	\begin{itemize}
		\item If $\gamma \geq 1$, then $t^{\gamma} - {t'}^{\gamma} \geq (t-t')^{\gamma}$. 
		\item If $\gamma \in (0,1)$, then $t^{\gamma} - {t'}^{\gamma} \geq \frac{\gamma}{T}(t-t')$.
	\end{itemize}
\end{lemma}
\begin{proof}
	The first is by convexity of $t \mapsto t^{\gamma}$ on $\R_{> 0}$ for $\gamma \geq 1$. The second is because $t^{\gamma} - {t'}^{\gamma} \geq \gamma(t-t')t^{\gamma - 1}$ by concavity of $t \mapsto t^{\gamma}$ on $\R_{>0}$ for $\gamma \leq 1$, and then bounding $t^{\gamma - 1} \geq T^{\gamma - 1} \geq T^{-1}$.
\end{proof}

\begin{lemma}\label{lem:collectionmonotonicity}
	Let $p \in [1, \infty)$ and $q \in (1, \infty)$. There exists a constant $\alpha' > 0$ that depends only on $p$ and $q$ and satisfies the following. If $x_1,\ldots,x_k \in \{0,1\}^d$ is a $(k,s,t)$-collection with $s \geq k-1$ and $t < \binom{k}{2}$, then there exists a $(k,s,t+1)$-collection $(x'_1,\ldots,x'_k) \subset \{0,1\}^d$ such that
	\[
	\val_{p,q}(x'_1,\ldots,x'_k) \leq \val_{p,q}(x_1,\ldots,x_k) - \delta,
	\]
	where $\delta = (kd)^{-\alpha'}$.
\end{lemma}
\begin{proof}
	By $t < \binom{k}{2}$ and item (iii) of Definition~\ref{def:kstcollection}, there is a pair of vectors that is supported on disjoint entries. Without loss of generality, these two vectors are $x_1$ and $x_2$. Note that $x_1$ and $x_2$ can each share one non-zero entry with at most $k-2$ other vectors. Since $s \geq k-1$, item (i) of Definition~\ref{def:kstcollection} implies that each of the vectors $x_1,x_2$ has a non-zero entry that is zero for every other vector. 
	Thus, after a possible permutation of the $d$ coordinates, we may assume without loss of generality that $[x_1]_i = \mathds{1}[i=1]$ and $[x_2]_i = \mathds{1}[i=2]$ for all $i \in [k]$. 
	
	We construct a $(k,s,t+1)$-collection $(x'_1,\ldots,x'_k)$ by letting $x'_i = x_i$ for all $i \geq 2$, and setting $x'_1$ to be $x_1$ where the first two entries are modified: $[x'_1]_1 = 0$ and $[x'_1]_2 = 1$. It is straightforward to verify that this is a $(k,s,t+1)$-collection: (i) each vector $x'_i$ has $s$ non-zeros, (ii) and (iii) the only inner product between a pair that has changed is that now $\langle x'_1,  x'_2 \rangle = 1$, whereas before $\langle x_1,  x_2 \rangle = 0$, and (iv) only $x'_1$ and $x'_2$ are non-zero on coordinate $2$.

	We now prove that $\val_{p,q}(x_1,\ldots,x_k)$ is larger than $\val_{p,q}(x'_1,\ldots,x'_k)$ by at least $\delta$. Let $y \in [0,1]^d$ be such that $\val_{p,q}(x_1,\ldots,x_k) = \sum_{i=1}^k \|x_i - y\|_q^p$, and construct $y' \in \RR^d$ by setting $y'_1 = 0$, $y'_2 = \min(1,((y_1)^q + (y_2)^q)^{1/q}) \in [0,1]$, and $y'_i = y_i$ for all $j \geq 3$. From this choice of $y'$, we have the relations $y_1, y_2 \leq y_2'$, which will be used in the sequel. Then on one hand,
	\begin{align*}
		\val_{p,q}(x_1,\ldots,x_k) 
		=
		\sum_{i=1}^k \|x_i - y\|_{q}^p
		=
		\underbrace{((1 - y_1)^q + (y_2)^q + c_1)^{\frac{p}{q}}}_{h_1} + \underbrace{((y_1)^q + (1 - y_2)^q + c_2)^{\frac{p}{q}}}_{h_2} + \underbrace{\sum_{j=3}^k ((y_1)^q + (y_2)^q + c_j)^{\frac{p}{q}}}_{h_3} 
	\end{align*}
	where $c_i = \sum_{j=3}^d |[x_i]_j - y_j|^q \in [0,d-2]$ for $i \in [k]$. And on the other hand,
	\begin{align*}
		\val_{p,q}(x'_1,\ldots,x'_k)
		\leq \sum_{i=1}^k \|x'_i - y'\|_q^p \nonumber
		= \underbrace{((1-y'_2)^q + c_1)^{\frac{p}{q}}}_{h_1'}
		+ \underbrace{((1-y'_2)^q + c_2)^{\frac{p}{q}}}_{h_2'} 
		+ \underbrace{\sum_{i=3}^k ((y'_2)^q + c_i)^{\frac{p}{q}}}_{h_3'}.
	\end{align*}
	Therefore it suffices to show that $h_2' \leq h_2$, $h_3' \leq h_3$, and $h_1' \leq h_1 - \delta$.
	\par To show $h_2' \leq h_2$, note that $(1-y_2')^q \leq (1-y_2)^q \leq y_1^q + (1-y_2)^q$ since $y_2 \leq y_2' \leq 1$. To show $h_3' \leq h_3$, note that the construction of $y'$ implies $(y_2')^q \leq (y_1)^q + (y_2)^q$. Finally, to show $h_1' \leq h_1 - \delta$, note that $h_1 = t^{p/q}$ and $h_1' = (t')^{p/q}$ where $t = (1-y_1)^q + y_2^q + c_1$ and $t' = (1-y_2')^q + c_1$. Observe that $0 \leq t' \leq t \leq d$ and in fact  $t - t' = (1-y_1)^q + y_2^q - (1-y_2')^q \geq y_2^q$ because $y_1 \leq y_2' \leq 1$. Thus by Lemma~\ref{lem:pq-helper}, we have that $h_1 - h_1' = t^{p/q} - {t'}^{p/q}$ is at least $y_2^p$ if $p \geq q$, and at least $(p/qd) y_2^q$ if $p < q$. In either case, this gap is at least inverse polynomially large in $k$ and $d$ by the analogous inverse polynomial lower bound on $y_2$ in Lemma~\ref{lem:yilowerbound}. Choosing a suitable constant $\alpha'$ completes the proof.
\end{proof}

We now reason about $(k,s,t)$-collections that could arise as the embedding of a $k$-clique (i.e., the case that $t = \binom{k}{2}$). We prove that all such $(k,s,t)$-collections have the same value.
\begin{lemma}\label{lem:cliquesequal}
	If $(x_1,\ldots,x_k)$ and $(x'_1,\ldots,x'_k)$ are both $(k,s,\binom{k}{2})$-collections, then $\val_{p,q}(x_1,\ldots,x_k) = \val_{p,q}(x'_1,\ldots,x'_k)$.
\end{lemma}
\begin{proof}
	By adding padding zeros, we may assume that $x_1,\ldots,x_k,x'_1,\ldots,x'_k \in \{0,1\}^d$ for some common dimension $d$. Observe that to be a $(k,s,\binom{k}{2})$ collection, the $k$ vectors are such that each pair of vectors shares exactly one non-zero entry (on which all other vectors are zero) and otherwise has disjoint support. Thus $(x_1,\ldots,x_k)$ equals $(x'_1,\ldots,x'_k)$ up to a permutation of the $d$ coordinates and the $k$ vectors. The claim follows since $\val_{p,q}(x_1,\ldots,x_k)$ is invariant with respect to padding zeros, permuting the vectors, and permuting coordinates.
\end{proof}

We bring the above lemmas together to prove Lemma~\ref{lem:gap} in the case $q \in (1,\infty)$:
\begin{proof}[Proof of Lemma~\ref{lem:gap} in the case of $q \in (1,\infty)$]
	We are given an $n$-vertex, $D$-regular graph $G = (V,E)$. We compute the embedding vectors $\{x_{i,j}\}_{i \in [k], j \in [n]}$ by letting $x_{i,j} = \phi(i,v_i) \in \{0,1\}^d$ for all $i \in [k], j \in [n]$. This can be computed in $\poly(n,k)$ time by using the formula \eqref{eq:embedding-q1inf}, since $d = \binom{k}{2}n^2 \leq \poly(n,k)$. Let $v_1,\ldots,v_k \in V$ be a sequence of not-necessarily-distinct vertices. By Lemma~\ref{lem:whykst}, $(x_{1,v_1},\ldots,x_{k,v_k})$ is a $(k,D(k-1), |E(v_1,\ldots,v_k)|)$-collection of vectors.
	
	Now suppose we have access to $x'_1,\ldots,x'_k \in \{0,1\}^{d'}$ which is a $(k,D(k-1),\binom{k}{2})$ collection of vectors. If $\{v_1,\ldots,v_k\}$ is a $k$-clique in $G$, then $|E(v_1,\ldots,v_k)| = \binom{k}{2}$, so by Lemma~\ref{lem:cliquesequal}:
	\begin{equation*}
		\val_{p,q}(x_{1,v_1},\ldots,x_{k,v_k}) = \val_{p,q}(x'_1,\ldots,x'_k)
	\end{equation*}
	On the other hand, if $\{v_1,\ldots,v_k\}$ is not a $k$-clique in $G$, then we prove that $\val_{p,q}(x_{1,v_1},\ldots,x_{k,v_k})$ is strictly larger than $\val_{p,q}(x'_1,\ldots,x'_k)$, with a polynomial-size gap. In this case, $|E(v_1,\ldots,v_k)| < \binom{k}{2}$. Using Lemma~\ref{lem:collectionmonotonicity}, it follows inductively on $t$ that for each $|E(v_1,\ldots,v_k)| < t \leq \binom{k}{2}$ there is a $(k,s,t)$-collection $z^{(t)}_1,\ldots,z^{(t)}_k \in \{0,1\}^d$ such that $\val_{p,q}(z^{(t)}_1,\ldots,z^{(t)}_k) \leq \val_{p,q}(x_{1,v_1},\ldots,x_{k,v_k}) - \delta$, where $\delta = (kd)^{-\alpha'}$ as in Lemma~\ref{lem:collectionmonotonicity}. Therefore, $\delta \geq (nk)^{-\alpha''}$ for a suitable constant $\alpha'' \geq 0$ that depends only on $p$ and $q$. Thus, for $T = \binom{k}{2}$, we have
	\begin{equation*}
		\val_{p,q}(x_{1,v_1},\ldots,x_{k,v_k}) \geq \val_{p,q}(z^{(T)}_1,\ldots,z^{(T)}_k) + \delta = \val_{p,q}(x'_1,\ldots,x'_k) + \delta,
	\end{equation*}
	where the equality is by Lemma~\ref{lem:cliquesequal}. Therefore, to conclude the proof of Lemma~\ref{lem:gap} in the case $q \in (1,\infty)$, it only remains to prove that $x'_1,\ldots,x'_k \in \mathbb{R}^{d'}$ exists and that $\val_{p,q}(x'_1,\ldots,x'_k)$ and $\delta$ can be efficiently approximated.
	
	We construct $x'_1,\ldots,x'_k \in \{0,1\}^{d'}$ by letting $d' = \binom{k}{2} + k(D-1)(k-1)$. We index the first $\binom{k}{2}$ coordinates by pairs $(i',i'')$ such that $i' < i'' \in [k]$, and let $[x'_i]_{(i',i'')} = 1$ if and only if $i \in \{i',i''\}$. Finally, the remaining $k(D-1)(k-1)$ coordinates are used to pad each of the vectors $x'_1,\ldots,x'_k$ with $(D-1)(k-1)$ ones that are disjoint from the other vectors. From the construction, all distinct pairs of vectors $x'_i$ and $x'_{i'}$ have inner product $\langle x'_i,x'_{i'} \rangle = 1$, all vectors have number of nonzero entries $\|x'_i\|_0 = D(k-1)$, and every entry is nonzero for at most two vectors. Thus, $x'_1,\ldots,x'_k$ is a $(k,D(k-1),\binom{k}{2})$-collection that we have constructed in $\poly(n,k)$ time.
	
	Finally, define $\Delta := (nk)^{-\alpha''} / 2$, which is chosen so that $\Delta \leq \delta / 2$. In $\poly(n,k)$ time, compute a value $\gamma$ such that $\val_{p,q}(x'_1,\ldots,x'_k) \leq \gamma \leq \val_{p,q}(x'_1,\ldots,x'_k) + \Delta$. This may be done via a standard, out-of-the-box convex optimization algorithm, since the value $\val_{p,q}(x'_1,\ldots,x'_k)$ is the solution to a convex optimization problem by Lemma~\ref{lem:pq-convexity}. We conclude that if $\{v_1,\ldots,v_k\}$ is a $k$-clique then:
	$$\val_{p,q}(x_{1,v_1},\ldots,x_{k,v_k}) \leq \gamma,$$ and otherwise
	$$\val_{p,q}(x_{1,v_1},\ldots,x_{k,v_k}) \geq \gamma + \Delta,$$ where $\gamma$ and $\Delta$ can be computed in $\poly(n,k)$ time. This proves Lemma~\ref{lem:gap} for $q \in (1,\infty)$.
\end{proof}

\subsection{Case $q = 1$}\label{app:q1}
We prove Lemma~\ref{lem:gap} in the case of $q=1$ and general $p \in [1,\infty)$. Let $G = (V,E)$ be an $n$-vertex, $D$-regular graph and let $k > 0$ be an even integer. A new embedding function $\psi$ is needed to prove the case $q = 1$, since the embedding $\phi$ defined in \eqref{eq:embedding-q1inf} fails.
\begin{remark}[Failure of the embedding in the $q = 1$ case]\label{rem:q1failure}
	The embedding $\phi$ defined in \eqref{eq:embedding-q1inf} cannot be used for the case $q = 1$, since for any $k > 4$ and vertices $v_1,\ldots,v_k \in V$, it holds that $\val_{p,1}(\phi(1,v_1),\ldots,\phi(k,v_k)) = k (D(k-1))^p$, which does not depend on the number of edges $|E(v_1,\ldots,v_k)|$ between the vertices $\{v_i\}_{i \in [k]}$. This formula follows from a calculation which shows that $y = \vec{0}$ is an optimal choice for the optimization problem \eqref{eq:F-def} defining $\val_{p,1}(\phi(1,v_1),\ldots,\phi(k,v_k))$. Hence, the embedding $\phi$ cannot be used to distinguish between $k$-cliques and non-$k$-cliques when $q = 1$.
\end{remark}

Therefore, we define a new embedding $\psi = \psi_G : [k] \times [n] \to \{-1,0,1\}^d$, where $d = 2\binom{k}{2}n^2$. Each entry is indexed by an element of 
\[
\Sigma = \big\{(\ell,u,\ell',u',s) : \ell < \ell' \in [k], \mbox{ and } u,u' \in [n] \cong V, \mbox{ and } s \in \{+1,-1\} \big\}.
\]
For shorthand, we denote an element $(\ell,u,\ell',u',s) \in \Sigma$ by $\sigma$. For $\sigma \in \Sigma$, set
\begin{align}
	[\psi(i,v)]_{\sigma} = \begin{cases}
		\tau(\ell,\ell',i), & i \not\in \{\ell,\ell'\} \\
		s, & (i,v) = (\ell,u) \\
		s, & (i,v) = (\ell',u') \mbox{ and } (u,u') \in E \\
		-s, & (i,v) = (\ell',u') \mbox{ and } (u,u') \not\in E \\
		0, & i \in \{\ell,\ell'\} \mbox{ and } (i,v) \not\in \{(\ell,u),(\ell',u')\}
	\end{cases}.
	\label{eq:embedding-q1}
\end{align}
where we define $\tau(\ell,\ell',i) = (-1)^{|[i] \sm \{\ell,\ell'\}|}$. The restriction to even values of $k$, the choice $\tau$, and the addition of the extra parameter $s$ are carefully crafted to ensure that the vectors are sufficiently symmetric so that the optimal solution $y \in \RR^d$ to the optimization problem \eqref{eq:F-def} for $\val_{p,1}(\psi(1,v_1),\ldots,\psi(k,v_k))$ satisfies $\|\psi(i,v_i) - y\|_1 = \|\psi(i',v_{i'}) - y\|_1$ for all $i,i' \in [k]$. This is crucial to our proof that $\val_{p,1}(\psi(1,v_1),\ldots,\psi(k,v_k))$ is minimized when $\{v_1,\ldots,v_k\}$ is a $k$-clique. To prove Lemma~\ref{lem:gap}, we establish the following key lemma.

\begin{lemma}\label{lem:qequalsonegap}
	Let $v_1,\ldots,v_k \in V$ be vertices that are not necessarily distinct, and let $t = |E(v_1, \dots, v_k)|$.
	Then $$\val_{p,1}(\psi(1,v_1),\ldots,\psi(k,v_k)) \geq k^{1-p}\left(nk(k-1)(nk - 2n + 2) - 4t\right)^p.$$
	Furthermore, the bound holds with equality if $t = \binom{k}{2}$.
\end{lemma}

The proof of this lemma is technically involved because it requires analyzing a convex optimization problem by hand. For space considerations, the proof is provided in Appendix~\ref{sm:q1}. We now show how Lemma~\ref{lem:qequalsonegap} implies Lemma~\ref{lem:gap}.

\begin{proof}[Proof of Lemma~\ref{lem:gap} for $q = 1$]
	Given the $D$-regular, $n$-vertex graph $G$, and even integer $k > 0$, and parameters $p \in [1,\infty)$, $q = 1$, let the embedding $\{x_{i,j}\}_{i \in [k], j \in [n]}$ be given by $x_{i,j} = \psi(i,j)$ for all $i \in [k], j \in V \cong [n]$. This is $\poly(n,k)$-time computable by using the formula \eqref{eq:embedding-q1}, since the dimension is also polynomial: $d \leq \poly(n,k)$. 
	
	Furthermore, by Lemma~\ref{lem:qequalsonegap}, if $\{v_1,\ldots,v_k\} \subset V \cong [n]$ is a $k$-clique in $G$, then $\val_{p,1}(x_{1,v_1},\ldots,x_{k,v_k}) = \gamma := k^{1-p}(nk(k-1)(nk-2n+2) - 2k(k-1))^p$, and otherwise $\val_{p,1}(x_{1,v_1},\ldots,x_{k,v_k}) \geq \gamma + \Delta$, where $\Delta := k^{1-p}(nk(k-1)(nk-2n+2) - 2k(k-1) + 4)^p - k^{1-p}(nk(k-1)(nk-2n+2) - 2k(k-1))^p \geq 1 / k^p$.
\end{proof}

\subsection{Case $q = \infty$}\label{app:qinf}

Here, we prove Lemma~\ref{lem:gap} in the case $q = \infty$.  Let $G = (V,E)$ be an $n$-vertex graph and let $k > 0$. Unfortunately, the embedding used for the case of $q \in (1,\infty)$ fails, as we now remark, so we need a new embedding.
\begin{remark}[Failure of the embedding in the $q = \infty$ case]\label{rem:qinffailure}
	The embedding $\phi$ defined in \eqref{eq:embedding-q1inf} cannot be used in the case $q = \infty$, since for any $v_1,\ldots,v_k \in V$ and $k \geq 3$, we have
	$\val_{p,\infty}(\phi(1,v_1),\ldots,\phi(k,v_k)) = k / 2^p$, since one can show that choosing $y \in \mathbb{R}^d$ with $y_{j} = \frac{1}{2}$ for all $j \in [d]$ is optimal. Thus, since $\val_{p,\infty}(\phi(1,v_1),\ldots,\phi(k,v_k))$ is a fixed constant, we cannot distinguish between non-cliques and cliques using the embedding $\phi$ from \eqref{eq:embedding-q1inf}.
\end{remark}

Instead, for the $q = \infty$ case, consider the following embedding $\xi = \xi_G : [k] \times [n] \to \{-1,0,1\}^d$, where $d = \binom{k}{2}n^2$. The entries of $\xi$ are indexed by tuples of the form $(\ell,u,\ell',u')$ where $\ell < \ell' \in [k]$ and $u,u' \in [n] \cong V$. Set
\begin{align}
	[\xi(i,v)]_{(\ell,u,\ell',u')} = \begin{cases} 0, & (i,v) \not\in {\{(\ell,u),(\ell',u')\}} \\
		1, & (i,v) = (\ell',u') \\
		1, & (i,v) = (\ell,u) \mbox{ and } (u,u') \in E  \\
		-1, & (i,v) = (\ell,u) \mbox{ and } (u,u') \not\in E \end{cases}
	\label{eq:embedding-qinf}
\end{align}

In the following two lemmas, let $v_1, \dots, v_k \in V$, be vertices that are not necessarily distinct, and denote $x_i = \xi(i,v_i)$ for $i \in [k]$.

\begin{lemma}\label{lem:qinf-clique}
	If $\{v_1,\ldots,v_k\} \subset V$ forms a clique of size $k$ in graph $G$, then
	\[
	\val_{p,\infty}(x_1, \dots, x_k) \leq k/2^p.
	\]
\end{lemma}
\begin{proof}
	Define $y^* \in \RR^d$ as follows. For $\ell < \ell' \in [k]$ and $u,u' \in [n]$, set
	$$[y^*]_{(\ell,u,\ell',u')} = \begin{cases} -1/2, & \mbox{if } \exists i \in [k] \mbox{ such that }[x_i]_{(\ell,u,\ell',u')} = -1 \\ 1/2, & \mbox{ otherwise}\end{cases}.$$
	In the first case, let $i \in [k]$ be such that $[x_i]_{(\ell,u,\ell',u')} = -1$. Thus, we must have $(i,v_i) = (\ell, u)$. We observe that in this case, $[x_{i'}]_{(\ell,u,\ell',u')} \in \{-1,0\}$ for all $i' \in [k]$. Otherwise there must $i' \in [k]$ such that $(i',v_{i'}) = (\ell',u')$. This implies that $(v_{i},v_{i'}) \notin E$, since $[x_i]_{(i,v_i,i',v_{i'})} = -1$, which contradicts the assumption that $\{v_1, \dots, v_k\}$ is a clique. On the other hand, in the second case, $[x_i]_{(\ell,u,\ell',u')} \in \{0,1\}$ for all $i \in [k]$.
	
	Thus $\|x_i - y^*\|_{\infty} \leq 1/2$ for all $i \in [k]$, hence 
	$
	\val_{p,\infty}(x_1, \dots, x_k)
	= \min_{y \in \R^d} \sum_{i=1}^k \|x_i - y\|_{\infty}^p
	\leq \sum_{i=1}^k \|x_i - y^*\|_{\infty}^p
	\leq k/2^p$.
\end{proof}

\begin{lemma}\label{lem:qinf-notclique} 
	Suppose $n \geq 3$. If $\{v_1, \dots, v_k\} \subset V$ does not form a clique of size $k$ in graph $G$, then
	\[
	\val_{p,\infty}(x_1, \dots, x_k) \geq 2 + (k-2)/2^p.
	\]
\end{lemma}
\begin{proof}
	Without loss of generality $(v_1,v_2) \not\in E$. We make two observations that hold for any $y \in \R^d$.
	\par First, note that $\|x_1 - x_2\|_{\infty} = 2$ because $[x_1]_{(1,v_1,2,v_2)} = -1$ and $[x_2]_{(1,v_1,2,v_2)} = 1$. Thus by Jensen's inequality on the convex function $y \mapsto \|x-y\|_{\infty}^p$ (Lemma~\ref{lem:pq-convexity}), we have
	\begin{align}
		\frac{\left\|x_1-y
			\right\|_{\infty}^p + \left\|x_2-y\right\|_{\infty}^p }{2}
		\geq
		\left\|\frac{x_1 - x_2}{2}\right\|_\infty^p
		= 1^p
		= 1.
		\label{eq:qinf-notclique:1}
	\end{align}
	\par Second, note that $\|x_i - x_{i'}\|_{\infty} \geq 1$ for all $i \neq i' \in [k]$ because for $u \in V \sm \{v_i,v_{i'}\}$, it holds that $[x_i]_{(i,u,i',v_{i'})} = 0$ and $[x_{i'}]_{(i,u,i',v_{i'})} = 1$. Thus by Jensen's inequality on the convex function $y \mapsto \|x-y\|_{\infty}^p$ (Lemma~\ref{lem:pq-convexity}), we have that
	\begin{align}
		\sum_{i=3}^{k} \|x_i - y\|_{\infty}^p
		&=
		\frac{1}{(k-3)}  \sum_{3 \leq i < i' \leq k} \left( \|x_i  - y\|_{\infty}^p + \|x_{i'} - y\|_{\infty}^p \right) 
		\nonumber
		\\ &\geq
		\frac{2}{k-3} \sum_{3 \leq i < i' \leq k} \left\|\frac{x_i - x_{i'}}{2}\right\|_\infty^p \nonumber
		\\ &\geq \frac{1}{k-3} \cdot (k-2)(k-3) \cdot \frac{1}{2^{p}} \nonumber
		\\ &= \frac{k-2}{2^p}.
		\label{eq:qinf-notclique:2}
	\end{align}
	\par Therefore, since~\eqref{eq:qinf-notclique:1} and~\eqref{eq:qinf-notclique:2} hold for any $y \in \R^d$, we conclude that $\val_{p,\infty}(x_1,\ldots,x_k)
	= \min_y \sum_{i=1}^k \|x_i - y\|_{\infty}^p
	\geq 2 + (k-2)/2^p$.
\end{proof}

\begin{proof}[Proof of Lemma~\ref{lem:gap} for $q = \infty$]
	Given the $D$-regular, $n$-vertex graph $G$, integer $k > 0$, and parameters $p \in [1,\infty), q = \infty$, let the embedding $\{x_{i,j}\}_{i \in [k], j \in [n]}$ be given by $x_{i,j} = \xi(i,j)$ for all $i \in [k], j \in V \cong [n]$. This is $\poly(n,k)$-time computable with the formula \eqref{eq:embedding-qinf}, since the dimension satisfies $d \leq \poly(n,k)$.
	
	By Lemma~\ref{lem:qinf-clique}, if $\{v_1,\ldots,v_k\} \subset V \cong [n]$ is a $k$-clique in $G$, then $\val_{p,\infty}(x_{1,v_1},\ldots,x_{k,v_k}) \leq \gamma := k / 2^p$. Otherwise, $\val_{p,\infty}(x_{1,v_1},\ldots,x_{k,v_k}) \geq \gamma + \Delta$, where $\Delta := (2 + (k-2) / 2^p) - k / 2^p \geq 1$.
\end{proof}

\section{Discussion and outlook}\label{sec:conc}

The hardness results shown in this paper demonstrate that, under standard complexity-theoretic assumptions, it is impossible to compute arbitrarily close approximations for the high-dimensional Wasserstein barycenter problem in polynomial time. This motivates an interesting research direction about understanding what properties of distributions enable efficient computation of Wasserstein barycenters.

\par A first candidate could be to require all distributions to be uniform discrete distributions. Unfortunately, this does not help from a computational complexity perspective, as shown in the following extension of our main results. The proof is deferred to Appendix~\ref{sec:inapproxuniformproof}.

\begin{theorem}\label{thm:inapproxuniform}
	The statement of Theorem~\ref{thm:inapproxext} holds even when the distributions $\mu_1,\ldots,\mu_k$ are restricted to be uniform on their support.
\end{theorem}

\par Nevertheless, there is a growing body of work that shows that other assumptions do help, both in theory and practice. For example, polynomial-time computation of Wasserstein barycenters is possible for certain parametric families of high-dimensional distributions such as Gaussian distributions, or more generally location-scatter families~\citep{Alt21bures}. There is also recent work that shows promising empirical results for computing barycenters of high-dimensional distributions if they are supported on low-dimensional manifolds or are represented by a convolutional neural network generative model~\citep{cohen2020estimating}. Other assumptions that might be interesting to investigate are if the input distributions $\mu_i$ are drawn from some generative process, or if the points in their supports lie in structured geometric configurations. Further understanding what commonly arising properties of distributions ensure efficient computation would have immediate impact on the many data-science applications of Wasserstein barycenters.

\section*{Acknowledgments} We are grateful to Victor-Emmanuel Brunel, Jonathan Niles-Weed, and Pablo Parrilo for stimulating conversations, and to the anonymous reviewers for their insightful comments which have greatly improved the clarity of the exposition.

\appendix

\section{Deferred proof details}
\subsection{Proof of Proposition~\ref{prop:clique-special}}\label{supp:clique}

$\NP$-hardness of the standard $\clique$ problem is shown in~\citep{karp1972reducibility}. To show that we can restrict to $G$ being a regular graph, note that~\citep{mohar2001face} shows that it is $\NP$-hard to find a maximum independent set in regular graphs. The claim follows since a (maximum) independent set in $G$ is a (maximum) clique in the complement graph. Furthermore, we may assume $k$ is even without loss of generality: otherwise, if $k$ is is odd, consider instead the graph $G'$ that is two copies of $G$ plus edges between all pairs of vertices that lie in the different copies of $G$. Note that $S$ is a $k$-clique in $G$ if and only if the two copies of $S$ together form a $(2k)$-clique in $G'$. Also, $G'$ is still regular since $G$ is.

\subsection{Proof of Lemma~\ref{lem:whykst}}\label{app:lem-whykst}

We verify the four properties of a $(k,D(k-1),|E(v_1,\ldots,v_k)|)$-collection.
\begin{itemize}
	\item[(i)] For any $i \in [k]$, we have \begin{align*}\|\phi(i,v_i)\|_0 &= |\{(i,v_i,\ell',u') : (v_i,u') \in E, i < \ell' \leq k\} \cup \{(\ell,u,i,v_i) : (u,v_i) \in E, 1 \leq \ell < i\}| \\
		&= D(k-i) + D(i-1) = D(k-1),\end{align*}
	since the graph is $D$-regular.
	\item[(ii)] For any $i < i' \in [n]$, if $(v_i,v_{i'}) \in E$ , then $\phi(i,v_i)$ and $\phi(i',v_{i'})$ share exactly one non-zero entry: the entry $(i,v_i,i',v_{i'})$.
	\item[(iii)] Otherwise, for any $i < i' \in [n]$, if $(v_i,v_{i'}) \not\in E$, then $\phi(i,v_i)$ and $\phi(i',v_{i'})$ share no common non-zero entries.
	\item[(iv)] For any entry $(\ell,u,\ell',u')$, if $[\phi(i,v_i)]_{(\ell,u,\ell',u')} = 1$ then $(i,v_i) \in \{(\ell,u),(\ell',u')\}$ by definition, so indeed we have $|\{i : [\phi(i,v_i)]_{(\ell,u,\ell',u')} \neq 0\}| \leq |\{\ell,\ell'\}| \leq 2$.
\end{itemize}

\subsection{Proof of Lemma~\ref{lem:qequalsonegap}}\label{sm:q1}

We first prove a helper lemma.
\begin{lemma}\label{lem:q1-tau}
	For any $\ell < \ell' \in [k]$, we have $\sum_{i \in [k] \setminus \{\ell,\ell'\}} \tau(\ell,\ell',i) = 0$.
\end{lemma}
\begin{proof}
	The sum over $i \in \{1, \dots, \ell-1\}$ equals $-\mathds{1}[\ell \text{ is even}]$. The sum over $i \in \{\ell+1, \dots, \ell'-1\}$ equals $(-1)^{\ell}\mathds{1}[\ell' - \ell \text{ is even}]$. The sum over $i \in \{\ell'+1, \dots, k\}$ equals $\mathds{1}[\ell' \text{ is odd}]$ since $k$ is even. There are four cases arising from whether $\ell,\ell'$ are even. In all cases, the total is readily checked to be $0$.
\end{proof}

\begin{proof}[Proof of Lemma~\ref{lem:qequalsonegap}]
	For shorthand, let $x_i = \psi(i,v_i) \in \{-1,0,1\}^d$ for $i \in [k]$. By Jensen's inequality on the convex function $t \mapsto t^p$ and separability of the $\ell_1$ norm, we have
	\begin{align}
		\val_{p,1}(x_1,\ldots,x_k)
		= \min_{y \in \R^d} \sum_{i=1}^k \|x_i - y\|^p_1 \geq\min_{y \in \R^d} k\left(\frac{1}{k}\sum_{i=1}^k \|x_i - y\|_1\right)^p 
		=
		k^{1-p} \left( \sum_{\sigma \in \Sigma} \min_{y_\sigma \in \R} \sum_{i=1}^k \big|[x_i]_{\sigma} - y_{\sigma} \big| \right)^p
		\label{eq:lem-q1:1}
	\end{align}
	We explicitly solve the latter univariate minimization over each coordinate $y_{\sigma}$ in closed form. To this end, let $T = \{(i,v_i) \}_{i \in [k]}$ and define the partition the index set $\Sigma
	=
	\bigcup_{(a,b,c,s) \in \{0,1\}^3 \times \{-1,1\}} A_{a,b,c,s}$
	where
	\[
	A_{a,b,c,s} = \big\{\sigma = (\ell,u,\ell',u',s) \in \Sigma \;:\; |\{(\ell,u)\} \cap T| = a, \; |\{(\ell',u')\} \cap T| = b, \; |\{(u,u')\} \cap E| = c\big\}.
	\]
	Note that if $\sigma = (\ell,u,\ell',u',s) \in A_{a,b,c,s}$, then
	\begin{align*}
		\sum_{i=1}^k |[x_i]_{\sigma} - y_{\sigma}|
		&= \sum_{i \in \{\ell,\ell'\}} |[x_i]_{\sigma} - y_{\sigma}| + \sum_{i \not\in \{\ell,\ell'\}} |\tau(\ell,\ell',i) - y_{\sigma}|  \\ 
		&= \sum_{i \in \{\ell,\ell'\}} |[x_i]_{\sigma} - y_{\sigma}| + (k/2 - 1)|1 - y_{\sigma}| + (k/2 - 1)|1 + y_{\sigma}| \\
		&= (a+bc)|s - y_{\sigma}| + b(1-c) | s + y_{\sigma}| + (2-a-b)|y_{\sigma}| + (k/2 - 1)(|1 - y_{\sigma}| + |1 + y_{\sigma}|).
	\end{align*}
	where above the second equality is by Lemma~\ref{lem:q1-tau}. Thus, by a direct calculation,
	\begin{equation}\label{eq:acdevalues}
		\min_{y_{\sigma} \in \R} \sum_{i=1}^k |[x_i]_{\sigma} - y_{\sigma}| = \begin{cases} k-2, & a+b=0 \mbox{\quad (achieved by } y_{\sigma} = 0 \mbox{)}\\
			k-1, & a+b=1 \mbox{\quad (achieved by } y_{\sigma} = 0 \mbox{)} \\
			k, & a+b=2, c = 0 \mbox{\quad (achieved by } y_{\sigma} = 0 \mbox{)} \\
			k-2, & a+b=2, c = 1 \mbox{\quad (achieved by } y_{\sigma} = s \mbox{)}
		\end{cases}
	\end{equation}
	Further, since $t = |E(v_1,\ldots,v_k)|$, we have \begin{equation}\label{eq:acdesizes}
		|A_{a,b,c,s}| = \begin{cases} \binom{k}{2}(n-1)^2, & a+b=0 \\
			2\binom{k}{2}(n-1), & a+b=1 \\
			\binom{k}{2} - t, & a+b=2, c = 0 \\
			t, & a+b=2, c = 1 \\
		\end{cases}
	\end{equation}
	Therefore by combining~\eqref{eq:lem-q1:1},~\eqref{eq:acdevalues}, and~\eqref{eq:acdesizes}, and simplifying, we conclude the desired bound
	\begin{align*}
		\val_{p,1}(x_1, \dots, x_k)
		&\geq k^{1-p}\left( \sum_{(a,b,c,s) \in \{0,1\}^3 \times \{-1,1\}} \sum_{\sigma \in A_{a,b,c,s}} \min_{y_{\sigma} \in \R} |[x_i]_{\sigma} - y_{\sigma}|\right)^p  
		\\ &= k^{1-p}2^p \left(\binom{k}{2}(n-1)^2 (k-2) + 2\binom{k}{2}(n-1)(k-1) + \left(\binom{k}{2}-t\right)k + t(k-2)\right)^p 
		\\
		&= k^{1-p}\left(nk(k-1)(nk - 2n + 2) - 4t\right)^p.
	\end{align*}
	
	\par Next, we show that this bound holds with equality when $t = \binom{k}{2}$. To do this, note that it suffices to explicitly construct $y^* \in \R^d$ satisfying
	\begin{align}
		\|x_i - y^*\|_1 = n(k-1)(nk - 2n + 2) - 2(k-1)
		\label{eq:q1-tight}
	\end{align}
	for each $i \in [k]$, since then plugging in $\val_{p,1}(x_1, \dots, x_k) = \min_{y \in \R^d} \sum_{i=1}^k \|x_i - y\|_1^p \leq \sum_{i=1}^k \|x_i - y^*\|_1^p$ finishes the proof. To this end, construct $y^* \in \R^d$ by setting $y_{\sigma}^* = s$ for all $\sigma \in A_{1,1,1,s}$ and $s \in \{+1,-1\}$, and $0$ elsewhere. We now verify~\eqref{eq:q1-tight}. The sparsity pattern of $y_{\sigma}^*$ implies
	\begin{align}
		\|x_i - y^*\|_1
		=
		\|x_i\|_1 + \sum_{s \in \{-1,1\}} \sum_{\sigma \in A_{1,1,1,s}} \big( \left|[x_i]_{\sigma} - s\right| - \left|[x_i]_{\sigma}\right| \big).
		\label{eq:q1-1}
	\end{align}
	The first term in~\eqref{eq:q1-1} is
	\begin{align}
		\|x_i\|_1 = 2\binom{k}{2}n^2 - 2 (k-1)n(n-1) = n(k-1)(nk-2n+2)
		\label{eq:q1-2}
	\end{align}
	by direct computation. To compute the second term in~\eqref{eq:q1-1}, observe that because $\{v_1,\dots,v_k\}$ forms a $k$-clique in $G$, the sum over $\sigma \in A_{1,1,1,s}$ is a sum over $\sigma = (\ell,v_{\ell}, \ell', v_{\ell'},s)$ for $\ell < \ell' \in [k]$. Consider two cases:
	\begin{itemize}
		\item If $i \in \{\ell, \ell'\}$, then $[x_{i}]_{\sigma} = s$, hence $|[x_i]_{\sigma} - s| - |x_i|_{\sigma} = |s - s| - |s| = -1$. 
		Therefore the contribution of this case to the second term in~\eqref{eq:q1-1} is
		\begin{align}
			\sum_{s \in \{-1,1\}} \sum_{\substack{\ell < \ell' \in [k] \\ \text{s.t. } i \in \{\ell,\ell' \}}} \big( \left|[x_i]_{\sigma} - s\right| - \left|[x_i]_{\sigma}\right| \big) 
			= 
			-2(k-1).
			\label{eq:q1-3}
		\end{align}
		\item Else if $i \notin \{\ell,\ell'\}$, then $[x_i]_{\sigma} = \tau(\ell,\ell',i)$, hence $|[x_i]_{\sigma} - s| - |x_i|_{\sigma} = |\tau(\ell,\ell',i)-s| -1$. Therefore the contribution of this case to the second term in~\eqref{eq:q1-1} is
		\begin{align}
			\sum_{s \in \{-1,1\}} \sum_{\substack{\ell < \ell' \in [k] \\ \text{s.t. } i \notin \{\ell,\ell' \}}} \big( \left|[x_i]_{\sigma} - s\right| - \left|[x_i]_{\sigma}\right| \big) 
			= 
			\sum_{\substack{\ell < \ell' \in [k] \\ \text{s.t. } i \notin \{\ell,\ell' \}}}
			\left(|\tau(\ell,\ell',i)+1| + |\tau(\ell,\ell',i)-1| - 2 \right)
			= 0
			\label{eq:q1-4}
		\end{align}
		where the last step is because $|1+z| + |1-z| = 2$ for $z \in \{-1,1\}$.
	\end{itemize}
	Combining~\eqref{eq:q1-1},~\eqref{eq:q1-2},~\eqref{eq:q1-3}, and~\eqref{eq:q1-4} now yields the desired identity~\eqref{eq:q1-tight}.
\end{proof}

\subsection{Proof of Theorem~\ref{thm:inapproxuniform}}\label{sec:inapproxuniformproof}
Here we prove Theorem~\ref{thm:inapproxuniform}, which is the extension of our main approximation hardness result, Theorem~\ref{thm:inapproxext}, to the case where the input distributions are additionally restricted to be uniform. This theorem is restated fully below for convenience. Recall that $R_{p,q} = \max_{x,x' \in \cup_{i \in [k]} \mathrm{supp}(\mu_i)} \|x-x'\|_q^p$ denotes the $p$-th power of the $\ell_q$-norm diameter of the supports of the input measures.
\begin{theorem}
	Fix $p \in [1,\infty)$ and $q \in [1,\infty]$. Assuming $\NP \not\subset \BPP$, there does not exist a randomized algorithm that given uniform distributions $\mu_1,\ldots,\mu_k$ and weights $\lambda_1,\ldots,\lambda_k = 1/k$, computes the value of the Wasserstein barycenter problem \eqref{eq:baryext} to $\eps$ additive accuracy with probability at least $2/3$ in $\mathrm{poly}(n,k,d,\log U, R_{p,q} / \eps)$ time. 
\end{theorem}
\begin{proof}
	Suppose we are given arbitrary discrete measures $\mu_1,\ldots,\mu_k$, each supported on $n$ points in $\{x \in \RR^d : \|x\|_q \leq 1\}$.\footnote{The general case of larger (non-constant) $R_{p,q}$ follows from the scale-invariance of the quantity $ R_{p,q} / \eps$.}
	We make the following claim: there exist discrete measures $\mu'_1,\dots,\mu'_k$ that (i) are each uniform over $N$ points in $\{x \in \RR^d : \|x\|_q \leq 1\}$ where $N \leq \poly(n,k,1/\eps)$, (ii) are  $\poly(N,d,\log U)$-time computable, and 
	(iii) preserve the barycenter functional to $\eps$ additive error in the sense that
	\begin{align}
		\abs{\sum_{i=1}^k \lambda_i \cW_{p,q}^p(\mu_i,\nu) - \sum_{i=1}^k \lambda_i \cW_{p,q}^p(\mu_i',\nu)} \leq \eps,
		\label{eq:unif-preserve}
	\end{align}
	for any measure $\nu$ supported on $\{x \in \RR^d : \|x\|_q \leq 1\}$.\footnote{I.e., for any $\nu$ that could be a candidate barycenter.}  The proof of the theorem follows by combining the claim with Theorem~\ref{thm:inapproxext}, because the claim allows one to reduce the problem of approximating the value of the barycenter of $\mu_1,\ldots,\mu_n$ to the problem of approximating the value of the barycenter of the uniform measures $\mu'_1,\ldots,\mu'_n$.
	\par We now give a proof of the claim. The measures $\mu'_i$ can be explicitly constructed as follows in two steps. Let $N$ be a positive integer to be chosen later.
	\begin{enumerate}
		\item \textit{Quantize}: The first step is to construct ``quantized'' measures $\tilde{\mu}_1,\ldots,\tilde{\mu}_k$.
		Denote the atoms of $\mu_i$ by $x_{i,j} \in \R^d$, and let $\mu_{i,j} \in [0,1]$ denote the corresponding masses. Define the distribution $\tilde{\mu}_i$ to have the same atoms $x_{i,j} \in \RR^d$, but now with masses $\tilde{\mu}_{i,j}$, chosen as follows: for every $i \in [k]$ and $j \in [n]$, quantize $\mu_{i,j}$ by rounding it to a multiple of $1/N$, i.e., choosing $\tilde{\mu}_{i,j} \in \{\lfloor \mu_{i,j} N\rfloor / N, \ceil{\mu_{i,j} N} / N\}$ for each $i \in [k], j \in [n]$, so that $\sum_{j \in [n]} \tilde{\mu}_{i,j} = 1$ for each $i \in [k]$.
		
		\item \textit{Split}: Now we may construct $\mu_1',\ldots,\mu'_k$ from $\tilde{\mu}_1,\ldots,\tilde{\mu}_k$ as follows: for any atom $x_{i,j}$ with $m/N$ mass, split it into $m$ distinct atoms $x_{i,j,1},\ldots,x_{i,j,m} \in \{x \in \RR^d : \|x\|_q \leq 1\}$ that are at distance $\|x_{i,j,\ell} - x_{i,j}\|_q \leq \eps/(p 2^{p})$ from the original atom, for each $\ell \in [m]$, and each have $1/N$ mass.
	\end{enumerate}
	
	\par To analyze step 1, consider a coupling between $\mu_i$ and $\tilde{\mu}_i$ given by the rounding transformation, where at most $1/N$ mass is moved for each of the $n$ atoms. Thus, $\cW_{p,q}(\mu_i,\tilde{\mu}_i)$ is at most the moved mass, which is at most $n/N$, times the $\ell_q$ diameter of the supports, which is at most $2$. 
	Thus by the triangle inequality, 
	\[
	\abs{\cW_{p,q}(\mu_i,\nu) - \cW_{p,q}(\tilde{\mu}_i,\nu)}
	\leq
	\cW_{p,q}(\mu_i,\tilde{\mu}_i)
	\leq 2n/N.
	\]
	To analyze step 2, we note that by the triangle inequality $$|\cW_{p,q}(\tilde{\mu}_i,\nu) - \cW_{p,q}(\mu'_i,\nu)| \leq \cW_{p,q}(\tilde{\mu}_i, \mu'_i) \leq \eps/(p 2^{p}),$$ where the second inequality holds because each atom is moved by at most distance $\eps/(p 2^{p})$ when constructing $\mu'_i$ from $\tilde{\mu}_i$. 
	So, overall, the triangle inequality gives $$|\cW_{p,q}(\mu_i,\nu) - \cW_{p,q}(\mu'_i,\nu)| \leq 2n/N + \eps/(p 2^{p}).$$
	Since $\nu$ is supported on $\{x : \|x\|_q \leq 1\}$, it follows that $\cW_{p,q}(\mu_i,\nu),\cW_{p,q}(\mu'_i,\nu) \in [0,2]$, thus
	\begin{align*}
		\abs{\cW_{p,q}^p(\mu_i,\nu) - \cW_{p,q}^p(\mu'_i,\nu)}
		&\leq 
		p 2^{p-1} \cdot \abs{\cW_{p,q}(\mu_i,\nu) - \cW_{p,q}(\mu'_i,\nu)} \\
		&\leq p 2^p \cdot n / N + \eps / 2 \\
		&\leq \eps. 
	\end{align*}
	In the last step, we take $N = \Theta(n p 2^p / \eps)$. This establishes~\eqref{eq:unif-preserve}, as desired.
\end{proof}

\small
\bibliographystyle{siamplain}
\bibliography{bary_hard}

\end{document}